\documentclass[12pt]{article}
\usepackage{graphicx} 

\usepackage{appendix}

\usepackage{amsmath,amssymb,amsthm}
\usepackage{amsrefs}
\usepackage{enumitem}
\usepackage{xcolor}
\usepackage{tikz}             
\usepackage{hyperref}
\usepackage{lineno}
\usepackage[margin=1in]{geometry}

\usepackage{endnotes}

\numberwithin{equation}{section}

\hypersetup{colorlinks = false, linkcolor = blue, urlcolor = blue, citecolor = blue}

\usepackage{lineno}
\modulolinenumbers[5]

\theoremstyle{plain}
\newtheorem{theorem}{Theorem}[section]

\newtheorem{lemma}[theorem]{Lemma}

\newtheorem{corollary}[theorem]{Corollary}

\theoremstyle{definition}

\theoremstyle{remark}

\title{Anisotropic mean curvature flow with contact angle and Neumann boundary conditions in arbitrary dimensions}
\date{}

\begin{document}
\author{Can Cui, Nung Kwan Yip \\
Department of Mathematics, Purdue University, West Lafayette, 47907 \\
cui147@purdue.edu, yipn@purdue.edu
}

\maketitle

\begin{abstract}
Over a bounded strictly convex domain in $\mathbb{R}^n$ with smooth boundary, we establish a priori gradient estimate for an anisotropic mean curvature flow with prescribed contact angle and Neumann boundary conditions. The estimates require careful analysis of the degeneracy property of the anisotropic mean curvature
operator. As a result, for both problems, we can infer that the solutions converge to one that is translation invariant in time. \\

\end{abstract}

\section{Introduction} \label{sec:1}
Mean curvature flow is a well known evolution equation for a hypersurface $M(t) \subset \mathbb{R}^{n+1}$ in which each point $X(t)\in M(t)$ moves with a velocity given by the mean curvature vector $\Vec{H}$:
\begin{equation} \label{eq:para_mcf}
\frac{\partial X}{\partial t}=\Vec{H}(X,t).
\end{equation} 
The concept of mean curvature is used in a large variety of applications, from modeling the behavior of  interfaces in materials,  formation of microstructures in material science to image processing and computer graphics. See \cite{Mu}, \cite{Ta}, \cite{AGLM} and \cite{desbrun1999implicit} for some expositions of these applications.

In this paper, we concentrate on the case when $M(t)$ is given by the graph of a function $u$ over a domain $\Omega\subseteq\mathbb{R}^{n}$: $M(t) = \{(x,u(x,t)): x\in\Omega\}$.
Then \eqref{eq:para_mcf} can be represented by the following
non-parametric form:
\begin{equation}\label{eq:graph_mcf}
     \frac{\partial u}{\partial t} = \left( \delta_{ij} -\frac{u_{x_i} u_{x_j}}{1+|Du|^2} \right) \frac{\partial^2 u}{\partial x_i \partial x_j}.
\end{equation}
(In the above, $1\leq i, j\leq n$, and we have used the Einstein's convention of summation over repeated indices.)
This equation also arises as the $L^2$-gradient flow for the surface area functional:
\begin{equation*}
    E(u) = \int \sqrt{1+|Du|^2} dx,
\end{equation*}
in the sense that $\frac{\partial u}{\partial t} = -\nabla_{L^2} E(u)$.

Over the years, various results have been obtained for the mean curvature flow
\eqref{eq:para_mcf} and \eqref{eq:graph_mcf} -- see for example \cites{EH1, Ecker2004}. 
In this paper, we focus on some boundary value problems associated with \eqref{eq:graph_mcf}. In particular, we analyze the contact angle and Neumann boundary conditions over a convex domain $\Omega$. 

The contact angle $\theta$ is the angle between the tangent plane of the function $u$ and the vertical plane over the domain boundary, $\partial\Omega$. Physically, if the graph of $u$ represents the surface of a liquid region and the solid phase is bounded by a vertical wall along $\partial\Omega$, then $\theta$ is the angle between the liquid and solid phases at their interface.
The contact angle thus describes how a liquid interacts with a solid surface. If $\theta < \pi/2$, then the liquid spreads out on the surface (wetting) while if $\theta > \pi/2$, then the liquid tends to form a droplet (non-wetting) - see Figure \ref{fig:1}. This is a result of the balance between the adhesive  (between the liquid and solid) and cohesive (within the liquid) forces. For more details on the physical background, see \cites{de2003capillarity, finn1986equilibrium}.
\begin{figure}[h] \label{fig:1}
\centering
\begin{minipage}{0.45\linewidth}
\includegraphics[width=1\linewidth]{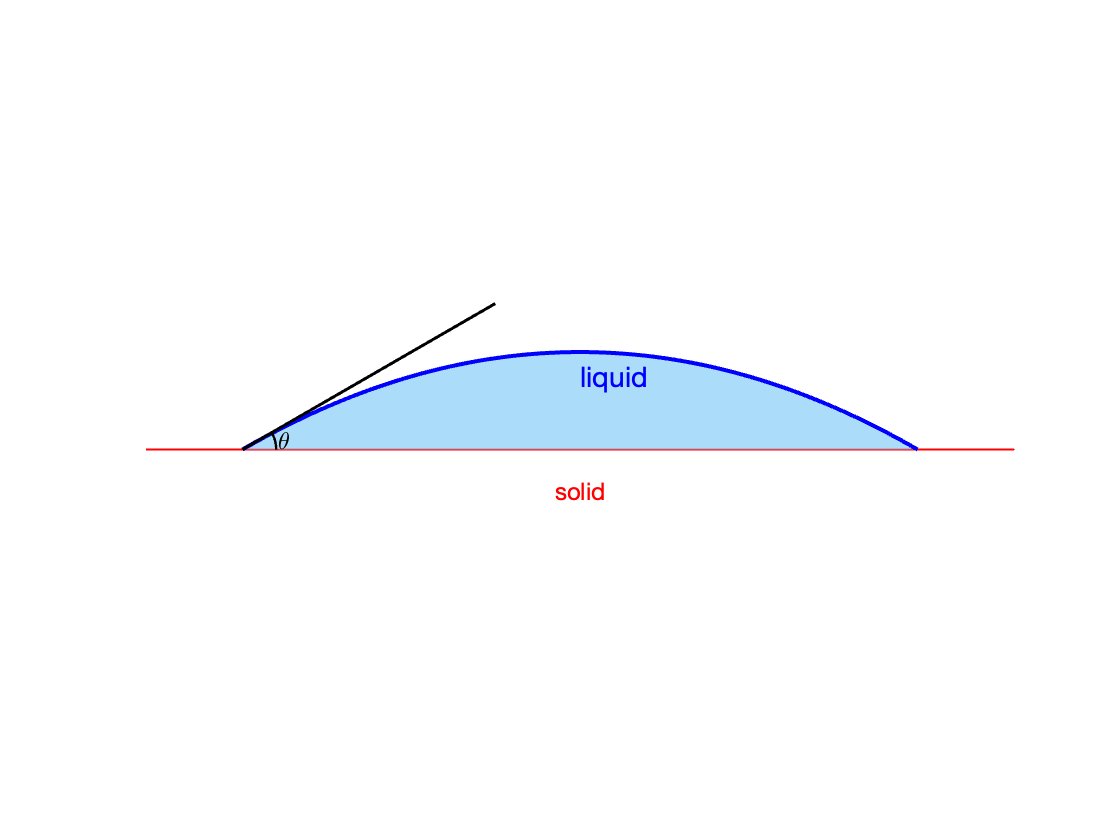}\\
\end{minipage}\hspace{20pt}
\begin{minipage}{0.5\linewidth}
\includegraphics[width=0.8\linewidth]{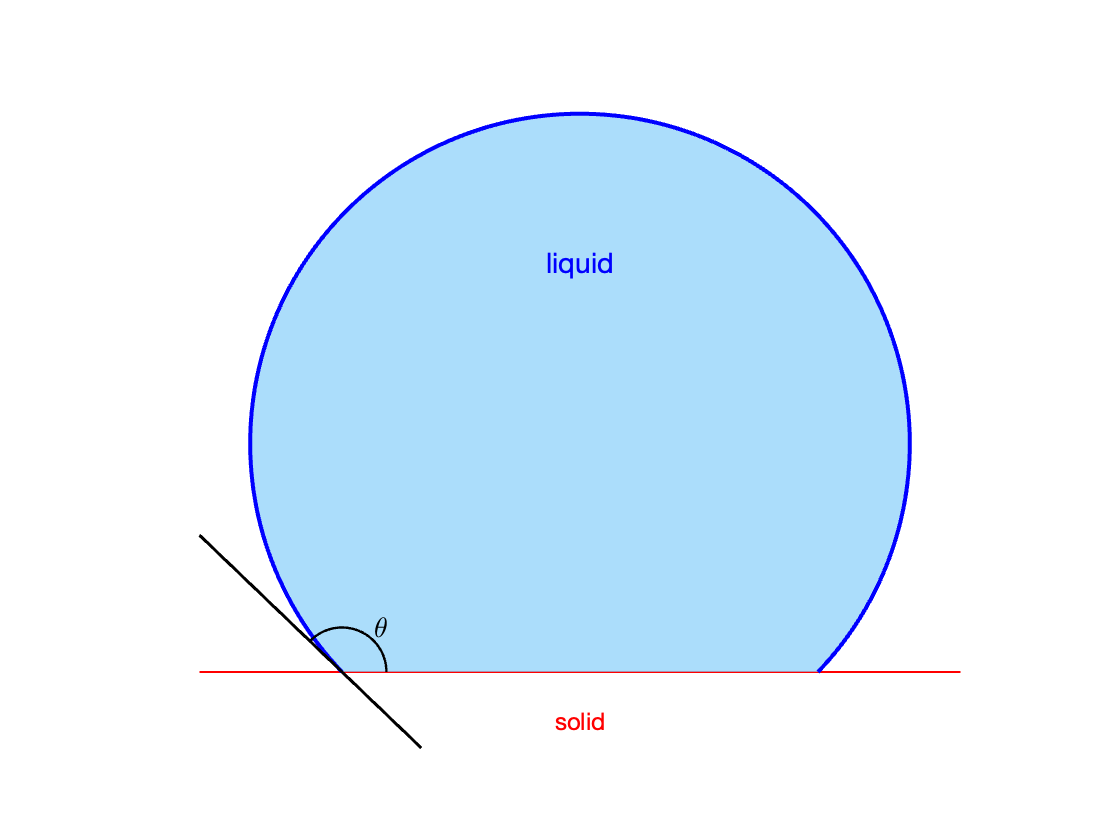}\\
\end{minipage}\hfill
\caption{  Left: $\theta<\pi/2$, liquid spreads out, Right: $\theta>\pi/2$, droplet.}
\end{figure}

The above description motivates the following contact angle boundary condition:
\begin{equation} \label{eq:hu_brypro}
D_N u = -  \cos \theta \sqrt{1+|Du|^2}  \quad   \text{on} \ \partial \Omega.    
\end{equation}
where $N$ is the inward unit normal to $\partial\Omega$.
In \cite{Hu1}, Huisken considered the above problem
with $\theta$ fixed to be $\frac{\pi}{2}$ essentially corresponding to a Neumann boundary condition. He proved that the solution of \eqref{eq:graph_mcf}-\eqref{eq:hu_brypro} converges to a constant function over $\Omega$. In \cite{AW1}, Altschuler and Wu considered \eqref{eq:graph_mcf}-\eqref{eq:hu_brypro}
when $\Omega$ has dimension one. They proved that if $\theta$ takes on a value
between $0$ and $\pi$, then the solution converges to a solution which translates vertically in time. This result also works for a class of quasilinear equations. Later, in \cite{AW2}, they proved the same convergence result over a domain in $\mathbb{R}^{2}$ with $\theta$ not necessarily a constant. However, their $\theta$ can not vary too fast along the domain boundary. In recent years, the work \cite{MWW} studied mean curvature flow with Neumann boundary condition and prove that solutions converge to ones moving by vertical translation. After that, \cite{GMWW} generalized Altschuler and Wu's result to a domain in higher dimensions, but $\theta$ is required to be sufficiently closed to $\frac{\pi}{2}$. 

The goal of this paper is to generalize the above results to anisotropic version of
\eqref{eq:hu_brypro}.  Similar to the isotropic case, it arises as the
gradient flows for the following area-like functional:
\begin{equation} \label{eq:ani_funcl}
    E(u) = \int F(-\nu)\sqrt{1+|Du|^2} dx,
\end{equation}
where $\nu = (-Du,1)/\sqrt{1+|Du|^2}$ is the upward normal to the graph, and $F(\cdot): \mathbb{R}^{n+1}\longrightarrow\mathbb{R}_+$ is a positive, convex and homogeneous function of degree one, in the sense that $F(\lambda p) = \lambda F(p)$ for all $\lambda \geq 0$ and $p\in\mathbb{R}^{n+1}$. Similar to the isotropic case (where $F(p)=|p|$), the non-parametric (graphical) anisotropic mean curvature flow is derived as:
\begin{equation}  \label{eq:nonpara_aniso}
    \frac{\partial u}{\partial t} = \sqrt{1+|Du|^2} \ D^2_{p_i p_j} F (Du,-1) \ \frac{\partial^2 u}{\partial x_i \partial x_j}.
\end{equation}
In the above, the derivatives $D_{p_ip_j}$ are taken with respect to the first $n$ components, $1\leq i, j \leq n$. In the rest of the paper, we write
\begin{equation}\label{eq:def_a}
    a^{ij} := \sqrt{1+|Du|^2} \ D^2_{p_i p_j} F (Du,-1).
\end{equation}
We mention here some previous work on anisotropic mean curvature flow \cite{An1} and \cite{Cl}. These work do not consider boundary conditions. The recent work \cite{CKN} also considers anisotropic flows but on the entire domain $\mathbb{R}^n$ and with a specific form of the anisotropy.

Anisotropic mean curvature flow has several applications where directional dependence plays a crucial role, distinguishing it from the isotropic case. In crystalline materials, the surface energy often varies by direction due to atomic arrangement. Hence it can accurately capture the evolution of crystal surfaces by accounting for direction-dependent surface energy -- see \cite{cahn1993equilibrium} and \cite{taylor1993motion}.
In image processing, as in \cite{alvarez1992image} and \cite{sochen1998general}, anisotropic flow enables edge-preserving filtering by smoothing noise in specific directions without blurring important edges. 
In fluid dynamics, anisotropic flow models fluid motions in porous or fibrous materials where permeability varies by direction. See \cite{garcke1998anisotropic} as a reference.

Now we state precisely the questions analyzed in this paper. Consider the following two boundary value problems over a bounded and strictly convex domain $\Omega \subset \mathbb{R}^{n}$:
 \begin{enumerate}
\item Prescribed contact angle boundary problem:
\begin{equation} \label{eq:angpro}
    \left\{ \begin{array}{cllcl}
            u_t & = & \sqrt{1+|Du|^2} \  D^2_{p_i p_j} F (Du,-1) u_{x_i x_j}   \quad &  \text{in} & \ Q_T = \Omega \times [0,T], \\
            D_N u & = & -  \cos \theta \sqrt{1+|Du|^2}  \quad &  \text{on} & \ \Gamma_T=\partial \Omega \times [0,T],  \\
         u(\cdot,0) & = & u_0(\cdot)  \quad \ & \text{in} & \ \Omega_0 = \Omega \times\{0\}.
    \end{array}  \right.
\end{equation}

\item Neumann boundary condition:
\begin{equation} \label{eq:neupro}
    \left\{ \begin{array}{cllcl}
            u_t & = & \sqrt{1+|Du|^2} \  D^2_{p_i p_j} F (Du,-1) u_{x_i x_j}   \quad &  \text{in} & \ Q_T, \\
            D_N u & = & \varphi  \quad &  \text{on} & \ \Gamma_T,  \\
         u(\cdot,0) & = & u_0(\cdot)  \quad \ & \text{in} & \  \Omega_0.
    \end{array}  \right.
\end{equation}
This can be viewed as a linearized version -- when $Du$ is small -- of the first contact angle problem.
\end{enumerate}
 
The main results are a priori gradient estimates for the above boundary value problems. 
These in turn prove the long time behavior of the solutions -- they converge to a 
solution that is translation invariant in time. They thus generalize earlier similar results
in the anisotropic case. Our main assumption is that the anisotropic function $F$ does not deviate too much from the isotropic function $\overline{F}(p)=|p|$. See Section \ref{sec:2} for the precise statements on $F$.

In our parallel work \cite{CY1}, we solve \eqref{eq:angpro} for a domain $\Omega \subseteq\mathbb{R}^2$.
To derive a boundary gradient bound in that case, we differentiated the left-hand side of the governing equation along the tangential and normal directions at a point on $\partial\Omega$. When $n=2$, we could determine the second-order tangential derivative of $u$ by differentiating the boundary condition in \eqref{eq:angpro}. By combining this calculation with maximum principle, we established a gradient bound of $u$.
However, when $n>2$, mixed second-order tangential derivatives arise, and they cannot be resolved by the same method. This issue also occurs when solving \eqref{eq:neupro}. To address this, we would need a new approach to derive the gradient estimate in higher dimensions. See Theorem \ref{thm:3.1} and Theorem \ref{thm:4.1} for our main results. 

It turns out that the same techniques can also be applied to prove the gradient estimate of the solutions to the following elliptic version \eqref{eq:angpro} and \eqref{eq:neupro},
\begin{equation} \label{eq:elang_pro}
\left\{ \begin{array}{cllcl}
     \lambda & = & \sqrt{1+|Dw|^2} \ D^2_{p_i p_j} F (Dw,-1) w_{x_i x_j}    
     \quad & \text{in} & \ \Omega,  \\
     D_N w & = & - \sqrt{1+|Dw|^2} \cos \theta \quad & \text{on} & \ \partial \Omega ,
\end{array} \right.
\end{equation}
and 
\begin{equation} \label{eq:elneu_pro}
    \left\{ \begin{array}{cllcl}
           \lambda  & = & \sqrt{1+|Dw|^2} \  D^2_{p_i p_j} F (Dw,-1) w_{x_i x_j}    \quad &  \text{in} &  \ \Omega , \\
            D_N w & = & \varphi  \quad &  \text{on} & \ \partial \Omega ,
    \end{array}  \right..
\end{equation}
Once these are established, by following a similar process as in \cite{AW2}, \cite{GMWW}, and \cite{MWW}, we obtain the asymptotic behavior of solutions to \eqref{eq:angpro} and \eqref{eq:neupro} which shows that
any solution converges to a vertically translating solution.

\section{Preliminary} \label{sec:2}
In this section, we provide some useful properties of the anisotropic function $F$.

\subsection{Estimates of homogeneous function} \label{sec:2.1}
In this section, we focus on properties of homogeneous functions on $\mathbb{R}^{n+1}$. 
Recall that our anisotropic function $F$ is assumed to be positive, convex and homogeneous of degree one. For the purpose of distinction, we will use $\overline{F}$ to denote the isotropic function 
$\overline{F}(p) = |p|$.

Throughout this section, the notation $\Big|_{p}$ means evaluation of functions at $p$.
For example, $F\Big|_{p}=F(p)$.
We will also decompose a vector $p$ in $\mathbb{R}^{n+1}$ as $p = (p' , p^{n+1} )$ where
$p'\in\mathbb{R}^n$ and $p^{n+1}\in\mathbb{R}$. From the homogeneous property of the function $F$, we have
\begin{equation} \label{hom_org}
         \ F\Big|_{p} = DF\Big|_{p} (p),
\end{equation}
where $D$ refers to taking derivative in a function's argument. Hence $DF\Big|_{p}(\cdot)$ is a linear form on $\mathbb{R}^{n+1}$. Similarly $D^2 F\Big|_{p}(\cdot,\cdot)$ and $D^3 F\Big|_{p}(\cdot,\cdot,\cdot)$ are respectively a bilinear form and a 3-way tensor on $\mathbb{R}^{n+1}$. We remark that $DF\Big|_{p}$ is an $(n+1)$-dimensional vector and $D^2 F\Big|_{p}$ is an $(n+1)\times(n+1)$ matrix. Referring to the notation in equation \eqref{eq:nonpara_aniso}, we use $\big[D^2_{p_i p_j} F\big]$ to represent the first $n\times n$ submatrix of $D^2F$.

For the rest of this section, we will let $p=(p',-1)$ be a fixed vector in $\mathbb{R}^{n+1}$. Due to the homogeneity of $F$, similar to \eqref{hom_org}, we have
\begin{equation} \label{homc2}
    D^2 F\Big|_p (p,\cdot) = D^2 F\Big|_p (\cdot,p) = 0,
\end{equation}
\begin{equation} \label{homc3}
    D^3 F\Big|_p (p,\cdot,\cdot) = -D^2 F\Big|_p (\cdot,\cdot).
\end{equation}

Next, we will state some properties of $D^2 F\Big|_p$ and $D^3 F\Big|_p$. For convenience,  we introduce some notations that will be used frequently later. Firstly, 
we choose an orthonormal basis of $\mathbb{R}^n$: 
\begin{equation}\label{ONB}
\{\phi^1, \cdots, \phi^n\}\,\,\,\text{with $\phi^n = \frac{p'}{|p'|}$ for $p'\neq0$.}
\end{equation}
Then we define
\begin{equation*}
    \tau^{\alpha \beta } := |p| \  D^2 F\Big|_{(p',-1)}(\phi^{\alpha}, \phi^{\beta}) \quad \text{for}\,\,\,\alpha,\beta = 1,\cdots,n.
\end{equation*}
Due to the homogeneity of $F$, we have $\tau^{\alpha \beta } \sim O(1)$. 
Similarly, for $D^3 F \Big|_p $, the following holds:
\begin{equation}\label{D3F.bd}
    D^3 F\Big|_{(p',-1)}(\phi^{\alpha}, \phi^{\beta}, \phi^{\gamma}) \sim O \left( \frac{1}{|p|^2}  \right)  \quad \alpha,\beta,\gamma = 1,\cdots,n.
\end{equation}
(In the above and for what follows, we use the convention that $a\sim O(b)$ means that there
exists a constant $C$ such that $|a| \leq C|b|$.)
However, the above generic property is not sufficient for our purpose. 
We need to have more elaborate estimates for the degeneracy of the anisotropic function $F$. This will be carried in the rest of this section.

Our mathematical approach requires us to understand the degeneracy of $D^2 F$. This is motivated by the
corresponding properties of the isotropic function $\overline{F}$ as seen in the following,
\begin{equation}
    |p|D^2 \overline{F} \Big|_{(p',-1)} (\phi^n,\phi^n) = \frac{1}{|p|^2} \sim O \left( \frac{1}{|p|^2}  \right) ,
\end{equation}
which is not of order $O(1)$ for $|p|\gg 1$ -- this is the main reason our operator is not uniformly elliptic. Similar results can also be stated for $D^3 \overline{F} \Big|_{(p',-1)}$.

Throughout our paper, we will assume the following symmetry condition for $F$:
\begin{equation}  \label{symc0}
    F\Big|_{(p',-1)} = F\Big|_{(p',1)}.
\end{equation}
This assumption also appears in \cite{Cl}.
The following properties are derived directly from \eqref{symc0}:
for any $q', r' \in \mathbb{R}^n $, we have,
\begin{equation} \label{symc1}
    DF\Big|_{(p',0)} (0,1) = 0,
\end{equation}
\begin{equation} \label{symc2}
    D^2 F\Big|_{(p',0)} ((0,1), (q',0)) = 0 ,
\end{equation}
\begin{equation} \label{symc3}
    D^3 F\Big|_{(p',0)} ((0,1), (q',0), (r',0)) =0,
\end{equation}
\begin{equation} \label{symc4}
    D^3 F\Big|_{(p',0)} ((0,1), (0,1), (0,1)) =0.
\end{equation}

With the above,  we have the following lemmas that quantify various degeneracy properties of $D^2 F$ and $D^3F$.
\begin{lemma} \label{lem:2.1}
    There exists a positive constant $C_1$ such that for any $q'\in \mathbb{R}^n$ with $|q'|=1$, we have
\begin{equation}\label{eq:lem2.1}
    \left|D^2 F\Big|_{(p',-1)} ((p',0), (q',0))\right| \leq \frac{C_1}{|p'|^2}.
\end{equation}
\end{lemma}
\begin{proof}
In the following we introduce a variable $s>0$ and $s'=\frac{1}{s}$ so that $s'\longrightarrow0$ as $s\longrightarrow\infty$.
Using equation \eqref{symc2}, we have: when $s\rightarrow \infty$, $s'=1/s\rightarrow 0$,
\begin{equation*}
\begin{split}
    & \ \ \ \  D^2 F\Big|_{(sp',-1)} ((sp',0), (q',0))  = D^2 F\Big|_{(\frac{p'}{s'},-1)} ((\frac{p'}{s'},0), (q',0)) \\
    & = - D^2 F\Big|_{(\frac{p'}{s'},-1)} ((0,-1), (q',0) ) = D^2 F\Big|_{(\frac{p'}{s'},-1)} ((0,1), (q',0) ) \\
    & = s' D^2 F\Big|_{(p',-s')} ((0,1), (q',0) ) \\
    & = s' \left( D^2 F\Big|_{(p',-s')} ((0,1), (q',0) ) - D^2 F\Big|_{(p',0)} ((0,1),(q',0))   \right) \\
    & = s'^2 \frac{D^2 F\Big|_{(p',-s')} ((0,1), (q',0) ) - D^2 F\Big|_{(p',0)} ((0,1),(q',0))}{s'}, 
\end{split}
\end{equation*}
so that 
\[
|sp'|^2D^2 F\Big|_{(sp',-1)} ((sp',0), (q',0))
=
|p'|^2 \frac{D^2 F\Big|_{(p',-s')} ((0,1), (q',0) ) - D^2 F\Big|_{(p',0)} ((0,1),(q',0))}{s'}
\]
which implies that
\begin{equation*}
    \lim_{s\rightarrow \infty} |sp'|^2 D^2 F\Big|_{(sp',-1)} ((sp',0), (q',0)) = |p'|^2 D^3 F\Big|_{(p',0)} ((0,-1), (0,1),(q',0)). 
\end{equation*}
Note that by \eqref{D3F.bd}, the right hand side of the above is bounded. Hence
\begin{equation*}
    \sup_{s\in[0,+\infty)} |sp'|^2 D^2 F\Big|_{(sp',-1)} ((sp',0), (q',0)) <
    \infty,
\end{equation*}
and \eqref{eq:lem2.1} follows by setting $s=1$ in the above.
\end{proof}

\begin{lemma} \label{lem:2.2}
There exists a positive constant $C_2$ such that for any $q',r'\in \mathbb{R}^n$ with $|q'|=|r'|=1$, we have
\begin{equation}
  \left|  D^3 F\Big|_{(p',-1)} ((0,1),(q',0),(r',0) ) \right| \leq \frac{C_2}{|p'|^3}.
\end{equation}
\end{lemma}
\begin{proof}
Using equation \eqref{symc3}, we have,
\begin{equation*}
\begin{split}
   & \ \ \ \  D^3 F\Big|_{(sp',-1)} ((0,1),(q',0),(r',0) ) = D^3 F\Big|_{(\frac{p'}{s'},-1)} ((0,1),(q',0),(r',0) ) \\
   & = s'^2 D^3 F\Big|_{(p', -s')} ((0,1),(q',0),(r',0) ) \\
   & = s'^2 \left(  D^3 F\Big|_{(p', -s')} ((0,1),(q',0),(r',0) ) -  D^3 F\Big|_{(p', 0)} ((0,1),(q',0),(r',0) ) \right) \\
   & = s'^3 \frac{D^3 F\Big|_{(p', -s')} ((0,1),(q',0),(r',0) ) -  D^3 F\Big|_{(p', 0)} ((0,1),(q',0),(r',0) ) }{s'},
\end{split}
\end{equation*}
and hence 
\begin{equation*}
    \lim_{s\rightarrow \infty} |sp'|^3 D^3 F\Big|_{(sp',-1)} ((0,1),(q',0),(r',0) ) = |p'|^3 D^4 F\Big|_{(p', 0)} ((0,-1),(0,1),(q',0),(r',0)),
\end{equation*}
so that the right hand side of which is bounded, i.e. 
\begin{equation*}
    \sup_{s\in[0,+\infty)} |sp'|^3 D^3 F\Big|_{(sp',-1)} ((0,1),(q',0),(r',0) ) < \infty.
\end{equation*}
Setting $s=1$ gives the desired result.
\end{proof}

\begin{lemma} \label{lem:2.3}
There exists a positive constant $C_3$ such that 
\begin{equation}
  \left|  D^3 F\Big|_{(p',-1)} ((0,1),(0,1),(0,1) \right| \leq \frac{C_3}{|p'|^3}.
\end{equation}
\end{lemma}
\begin{proof} The proof is very similar to the above.
Using equation \eqref{symc4}, we have,
\begin{equation*}
\begin{split}
    & \ \ \ \  D^3 F\Big|_{(sp',-1)} ((0,1),(0,1),(0,1) ) = D^3 F\Big|_{(\frac{p'}{s'},-1)} ((0,1),(0,1),(0,1) ) \\
    & = s'^2 D^3 F\Big|_{(p',-s')} ((0,1),(0,1),(0,1) ) \\
    & = s'^3 \frac{ D^3 F\Big|_{(p',-s')} ((0,1),(0,1),(0,1) ) - D^3 F\Big|_{(p',0)} ((0,1),(0,1),(0,1) ) }{s'},
\end{split}
\end{equation*}
so that
\begin{equation*}
    \lim_{s\rightarrow \infty} |sp'|^3 D^3 F\Big|_{(sp',-1)} ((0,1),(0,1),(0,1) = |p'|^3 D^4 F\Big|_{(p',0)} ((0,-1),(0,1),(0,1),(0,1) ).
\end{equation*}
Again, the right hand side is bounded 
\begin{equation*}
    \sup_{s\in[0,+\infty)} |sp'|^3 D^3 F\Big|_{(sp',-1)} ((0,1),(0,1),(0,1)) < \infty.
\end{equation*}
The claim thus follows as before.
\end{proof}

Next we will introduce the following tensor of order three: 
\begin{equation} \label{eq:def:T3}
    T^3_{ijl} = -D_{p_l} \left(|p| D^2_{p_i p_j} F\right)\Big|_{(p',-1)} 
    = -|p| D^3_{p_i p_j p_l} F\Big|_{(p',-1)} - D^2_{p_i p_j} F\Big|_{(p',-1)} \frac{p_l}{|p|} .
\end{equation}
which will be used frequently in the next section.
Note the symmetry property of $T^3$: $T^3_{ijl} = T^3_{jil}$.
The following lemma gives various estimates of $T^3$, in particular, its 
degeneracy along certain directions.
\begin{lemma} \label{lem:2.4}
Recall the basis from \eqref{ONB} with $\phi^n = \frac{p'}{|p'|}$. There exist positive constants $C_4, C_5, C_6$ depending only on $F$ such that for any $1\leq \alpha,\beta \leq n-1$, we have 
\begin{equation*}
    |T^3 (\phi^{\alpha}, \phi^{\alpha}, \phi^{\beta})| \leq \frac{C_4}{|p|},
\end{equation*}
\begin{equation*}
     |T^3 (\phi^{\alpha}, \phi^{\alpha}, \phi^{n})|,\ |T^3 (\phi^{\alpha}, \phi^n, \phi^n)|, \  |T^3 (\phi^{n}, \phi^{n}, \phi^{\alpha})| \leq \frac{C_4}{|p|^3},
\end{equation*}
\begin{equation*}
    | T^3 (\phi^{\alpha}, \phi^n, \phi^{\beta})| \leq \frac{C_4}{|p|} \quad (\alpha\neq \beta),
\end{equation*}
\begin{equation*}
     \frac{C_5}{|p|} \leq T^3 (\phi^{\alpha}, \phi^n, \phi^{\alpha}) \leq \frac{C_6}{|p|},
\end{equation*}
\begin{equation*}
\frac{C_5}{|p|^3} \leq T^3 (\phi^{n}, \phi^{n}, \phi^n)  \leq \frac{C_6}{|p|^3}.
\end{equation*}
\end{lemma}
\begin{proof}
Firstly, from \eqref{homc3}, we calculate the following to apply the results from Lemma \ref{lem:2.1} -- \ref{lem:2.3} to $D^3 F$ evaluated on basis elements,
\begin{equation*}
\begin{split}
     & D^3 F\Big|_{(p',-1)} ((p',0),(q',0),(r',0) ) \\
     = & - D^2 F\Big|_{(p',-1)} ((q',0), (r',0)) + D^3 F\Big|_{(p',-1)} ((0,1),(q',0),(r',0) ). \\
\end{split}
\end{equation*}
In addition, we have
\begin{equation*}
\begin{split}    
   &  D^3 F\Big|_{(p',-1)} ((p', 0), (p', 0), (p', 0)) \\
    = & - D^2 F\Big|_{(p',-1)} ((p', 0), (p', 0)) + D^3 F\Big|_{(p',-1)} ((0, 1), (p', 0), (p', 0)) \\
    & D^3 F\Big|_{(p',-1)} ((0, 1), (p', 0), (p', 0))  \\
    = & D^3 F\Big|_{(p',-1)} ((0, 1), (p', -1), (p', -1)) + D^3 F\Big|_{(p',-1)} ((0, 1), (p', -1), (0, 1)) \\
    & + D^3 F\Big|_{(p',-1)} ((0, 1), (0, 1), (p', -1)) + D^3 F\Big|_{(p',-1)} ((0, 1), (0, 1), (0, 1)) \\
    = & -2 D^2 F\Big|_{(p',-1)} ((0, 1), (0, 1)) + D^3 F\Big|_{(p',-1)} ((0, 1), (0, 1), (0, 1)) \\
    = & -2 D^2 F\Big|_{(p',-1)} ((p', 0), (p', 0)) + D^3 F\Big|_{(p',-1)} ((0, 1), (0, 1), (0, 1)).
\end{split}
\end{equation*}
Hence
\begin{equation*}
\begin{split}
    &  D^3 F\Big|_{(p',-1)} ((p', 0), (p', 0), (p', 0)) \\
    = & -3 D^2 F\Big|_{(p',-1)} ((p', 0), (p', 0)) + D^3 F\Big|_{(p',-1)} ((0, 1), (0, 1), (0, 1)).
\end{split}
\end{equation*}

To conclude, we rewrite results from Lemma \ref{lem:2.1} to \ref{lem:2.3} with respect to the basis elements
\eqref{ONB} $\{\phi^1, \cdots, \phi^n\}$ as follows:
for $1 \leq \tilde{\alpha}, \tilde{\beta}\leq n$, it holds that
\begin{equation} \label{eq:tau_an}
    \tau^{\tilde{\alpha}n} = \tau^{n \tilde{\alpha} } = |p| \  D^2 F\Big|_{(p',-1)}(\phi^{\tilde{\alpha}}, \phi^n) \sim O \left( \frac{1}{|p|^2}    \right),
\end{equation}
\begin{equation}  \label{eq:D3F_nab}
    D^3 F\Big|_{(p',-1)} (\phi^n,\phi^{\tilde{\alpha}}, \phi^{\tilde{\beta}} ) =  \frac{1}{|p'|} \left( -D^2 F\Big|_{(p',-1)}(\phi^{\tilde{\alpha}}, \phi^{\tilde{\beta}}) + D^3 F\Big|_{(p',-1)}((0,1),\phi^{\tilde{\alpha}}, \phi^{\tilde{\beta}}) \right).
\end{equation}
In particular, we have
\begin{equation}  \label{eq:D3F_nnn}
    D^3 F\Big|_{(p',-1)} (\phi^n,\phi^n, \phi^n ) = \frac{1}{|p'|} \left( -3 D^2 F\Big|_{(p',-1)}(\phi^{n}, \phi^{n}) + \frac{1}{|p'|^2}  D^3 F\Big|_{(p',-1)}((0,1),(0,1),(0,1)) \right).
\end{equation}
From \eqref{eq:D3F_nab} and \eqref{eq:D3F_nnn}, we compute
\begin{eqnarray*}
T^3 (\phi^{\alpha}, \phi^{\alpha}, \phi^{n})    
     & = & - |p| D^3 F\Big|_{(p',-1)} (\phi^{\alpha}, \phi^{\alpha}, \phi^{n}) - \frac{|p'|}{|p|} D^2 F\Big|_{(p',-1)} (\phi^{\alpha}, \phi^{\alpha})    \\
    & = & - \frac{|p|}{|p'|} D^3 F\Big|_{(p',-1)} ((0,1),\phi^{\alpha}, \phi^{\alpha}) + \frac{1}{|p|^2|p'|}  \tau^{\alpha \alpha} \ , \\
    T^3 (\phi^{\alpha}, \phi^n, \phi^n) & = & - |p| D^3 F\Big|_{(p',-1)} (\phi^{\alpha}, \phi^n, \phi^n) - \frac{|p'|}{|p|} D^2 F\Big|_{(p',-1)} (\phi^{\alpha}, \phi^n) \\
     & = & - \frac{|p|}{|p'|} D^3 F\Big|_{(p',-1)} ((0,1), \phi^n, \phi^{\alpha}) + \frac{1}{|p|^2|p'|} \tau^{\alpha n} \ , \\
    T^3 (\phi^{n}, \phi^{n}, \phi^{\alpha}) & = & - |p| D^3 F\Big|_{(p',-1)} (\phi^n, \phi^n, \phi^{\alpha}) \\
    & = & - \frac{|p|}{|p'|} D^3 F\Big|_{(p',-1)} ((0,1), \phi^n, \phi^{\alpha}) + \frac{1}{|p'|} \tau^{\alpha n} \ , \\ 
    T^3 (\phi^{\alpha}, \phi^n, \phi^{\beta}) & = & -|p| D^3 F\Big|_{(p',-1)} (\phi^{\alpha}, \phi^n, \phi^{\beta})  \\
    & = &  - \frac{|p|}{|p'|} D^3 F\Big|_{(p',-1)} ((0,1),\phi^{\alpha}, \phi^{\beta}) + \frac{1}{|p'|} \tau^{\alpha \beta} \ , \\
   T^3 (\phi^{n}, \phi^{n}, \phi^n) & = &  -|p| D^3 F\Big|_{(p',-1)} (\phi^n, \phi^n, \phi^n ) - \frac{|p'|}{|p|} D^2 F\Big|_{(p',-1)} (\phi^n, \phi^n ) \\
   & = & \frac{|p|}{|p'|^3} D^3 F\Big|_{(p',-1)} ((0,1)), (0,1), (0,1)) + \frac{2|p|^2 + 1}{|p'| |p|^2} \tau^{n n} \ . 
\end{eqnarray*}
Then Lemmas \ref{lem:2.2} and \ref{lem:2.3}, estimate \eqref{eq:tau_an} and the
homogeneity of $F$ suggest that there exists a constant $C_4$ such that 
\begin{equation*}
    |T^3 (\phi^{\alpha}, \phi^{\alpha}, \phi^{\beta})| \leq \frac{C_4}{|p|},
\end{equation*}
\begin{equation*}
     |T^3 (\phi^{\alpha}, \phi^{\alpha}, \phi^{n})|,\ |T^3 (\phi^{\alpha}, \phi^n, \phi^n)|, \ |T^3 (\phi^{n}, \phi^{n}, \phi^{\alpha})| \leq \frac{C_4}{|p|^3},
\end{equation*}
\begin{equation*}
    | T^3 (\phi^{\alpha}, \phi^n, \phi^{\beta})| \leq \frac{C_4}{|p|} \quad (\alpha\neq \beta).
\end{equation*}
By the fact that $F$ is convex and homogeneous of degree one, we have that 
$ \tau^{\alpha \alpha}\geq C$ and $\tau^{n n}\geq \frac{C}{|p|^2|}$ for some positive constant $C$.
Hence there are constants $C_5$ and $C_6$ such that
\begin{equation*}
    \frac{C_5}{|p|} \leq T^3 (\phi^{\alpha}, \phi^n, \phi^{\alpha}) \leq \frac{C_6}{|p|},
\end{equation*}
\begin{equation*}
    \frac{C_5}{|p|^3} \leq T^3 (\phi^{n}, \phi^{n}, \phi^n)  \leq \frac{C_6}{|p|^3}
\end{equation*}
completing the proof.
\end{proof}

Combining the previous considerations, we can assert that for our anisotropic function $F$, 
there are positive constants $c_1,\ c_2$ and $c_3$ such that for any $1\leq \alpha,\beta \leq n-1$, with $\alpha\neq\beta$, the following statements hold:
\begin{equation} \label{eq:F_cond1}
     c_1 \leq\tau^{\alpha \alpha} =  |p| D^2 F\Big|_{(p',-1)} (\phi^{\alpha}, \phi^{\alpha}) \leq c_3 , \quad \frac{c_1}{|p|^2} \leq \tau^{n n} = |p| D^2 F\Big|_{(p',-1)} (\phi^n, \phi^n) \leq \frac{c_3}{|p|^2} ,
\end{equation}
and 
\begin{equation} \label{eq:F_cond2}
    |\tau^{\alpha \beta}| = |p|\left|D^2 F\Big|_{(p',-1)} (\phi^{\alpha}, \phi^{\beta})\right| \leq c_2, \quad 
    |\tau^{\alpha n}| = |p|\left| D^2 F\Big|_{(p',-1)} (\phi^\alpha, \phi^n)\right|\leq \frac{c_2}{|p|^2}.
\end{equation}
The first parts of \eqref{eq:F_cond1} and \eqref{eq:F_cond2} are due to the homogeneity property of $F$ while the second parts are from Lemma \ref{lem:2.1}. For later convenience, without loss of generality, we can choose $c_1$ and $c_2$ such that
\begin{equation} \label{eq:F_cond3}
   c_1 \leq C_5 \,\,\,\text{and}\,\, \, C_4 \leq c_2.
\end{equation}
Note that in the isotropic case, we can take $c_1 =c_3 = 1$ and $c_2 = 0$. 
Hence the ratio $c_2/c_1$ suggests a way to measure how far the anisotropic function $F$ deviates from the isotropic function.

\subsection{Application} \label{Sec:2.2}
In this section, we apply the previous properties of the anisotropic function 
$F$ to the calculations in the derivation of the gradient estimate. 
For what follows, for a smooth function $u$, we introduce
\begin{equation*}
    p'=Du,\,\,\,p=(Du,-1),\,\,\,v=|p|= \sqrt{1+|Du|^2}.
\end{equation*}
Now the orthonormal basis \eqref{ONB} becomes:
\begin{equation}
\label{ONB2}
    \phi^1, \cdots, \phi^{n-1}, \phi^n = \frac{Du}{|Du|},
    \quad\text{for $Du\neq 0$}.
\end{equation}
(Later on, we will in fact just handle the case $|Du|>1$.)
We remark that in the isotropic case $\overline{F}(p)=|p|$, upon writing
$G(Du) = \big(|p| \ D^2 \overline{F}\big) \Big|_{(Du,-1)} $, that is 
\begin{equation*}
    G(Du) = \left[ \delta_{ij} - \frac{u_{x_i} u_{x_j}}{1+|Du|^2} \right]_{n\times n} = \left( \begin{array}{ccc}
       1 - \frac{u_{x_1} u_{x_1}}{1+|Du^2|}  & \cdots & -\frac{u_{x_1} u_{x_n}}{1+|Du|^2}  \\
        \vdots & \ddots & \vdots \\
        -\frac{u_{x_n} u_{x_1}}{1+|Du|^2} & \cdots & 1 - \frac{u_{x_n} u_{x_n}}{1+|Du|^2}
    \end{array}    \right)\ ,
\end{equation*}
then $\phi^n$ is the eigenvector of $G$ with eigenvalue $\frac{1}{v^2}$, while any vector in the orthogonal subspace is an eigenvector of $G$ with eigenvalue $1$. 

Now the basis $\{\phi^1, \cdots, \phi^n\}$ of $\mathbb{R}^n$ defined earlier induces an orthonormal basis of $n \times n$ symmetric matrices with respect to the Frobenius norm -- $\|M\| = \sqrt{\text{tr}(M^TM)}$: 
\begin{equation}\label{eq:PhiBasis}
\begin{split}
    &  \Phi^{\alpha \alpha} = \phi^{\alpha} (\phi^{\alpha})^T, \ \Phi^{\alpha \beta} = \frac{1}{\sqrt{2}} \left( \phi^{\alpha} (\phi^{\beta})^T + \phi^{\beta} (\phi^{\alpha})^T \right), \quad 1\leq \alpha < \beta \leq n-1,  \\
    &  \Phi^{\alpha n} = \frac{1}{\sqrt{2}} \left( \phi^{\alpha} (\phi^{n})^T + \phi^{n} (\phi^{\alpha})^T \right), \ \Phi^{n n} = \phi^{n} (\phi^{n})^T, \quad 1 \leq \alpha \leq n-1.
\end{split}
\end{equation}
(Note that the number of elements above is $n(n+1)/2$ which is exactly the dimension of $n\times n$ symmetric
matrices.) With that, the Hessian matrix $D^2u$ of $u$ at any given point can be written as a linear combination of the above $\Phi^{(\cdot, \cdot)}$'s:
\begin{equation}\label{D2U.Decomp}
    D^2 u =  \sum_{1\leq \alpha\leq\beta\leq n-1} \gamma^{\alpha \beta} \Phi^{\alpha \beta} + \sum_{1\leq\alpha\leq n-1} \gamma^{\alpha n} \Phi^{\alpha n} + \gamma^{n n} \Phi^{n n}.
\end{equation}

For the rest of this paper, we will make the further choice that 
$\big\{\phi^\alpha\big\}_{1\leq\alpha\leq n-1}$ are the eigenvectors of $D^2u$ (orthogonal to $\phi^n$).
Then the terms $\gamma^{\alpha \beta}$ in \eqref{D2U.Decomp} will vanish for $\alpha \neq \beta$ so that we have
\begin{equation*}
    D^2 u = \sum_{\alpha=1}^{n-1} \gamma^{\alpha \alpha} \Phi^{\alpha \alpha} + \sum_{\alpha=1}^{n-1} \gamma^{\alpha n} \Phi^{\alpha n} + \gamma^{n n} \Phi^{n n}.
\end{equation*}

Next, we want to compute the expression $T^3(D^2 u, V )$ for some given vector 
$$V =\sum_{1\leq\beta\leq n-1} \eta^{\beta} \phi^{\beta} + \eta^n \phi^n$$
where the functions $\eta^{\beta}$ are defined by
\begin{equation*}
    \eta^{\beta} = \gamma^{\beta \beta} \rho^{\beta} + \frac{ \gamma^{\beta n} }{\sqrt{2}} \rho^n,\,\,\,\text{for $1\leq \beta\leq n-1$, and}\,\,\,
    \eta^n = \sum_{1\leq \delta\leq n-1} \frac{\gamma^{\delta n}}{\sqrt2} \rho^{\delta}  + \gamma^{n n} \rho^n ,  
\end{equation*}
with $\rho^{\beta}$ and $\rho^n$ to be given in 
Section \ref{sec:ang_grad_case_2} \eqref{eq:ang_rho_def} and Section \ref{sec:4} \eqref{eq:neu_rho_def}.
Note that by means of definition \eqref{eq:def:T3} and the symmetry property of
$T^3$, we can abuse the notation $T^3(u_1\otimes u_2, u_3)$ as
$T^3(u_1, u_2, u_3)$.
With this in mind, we can write
\begin{equation}
    T^3(D^2 u, V ) = T^3\Big( \sum_{\alpha} \gamma^{\alpha \alpha} \Phi^{\alpha \alpha} + \sum_{\alpha} \gamma^{\alpha n} \Phi^{\alpha n} + \gamma^{n n} \Phi^{n n} , \sum_{\beta} \eta^{\beta} \phi^{\beta} + \eta^n \phi^n \Big).
\end{equation}
We compute this by evaluating $T^3$ on the basis elements $(\phi^1, \cdots, \phi^n)$ leading to the following decomposition:
\begin{equation} \label{eq:T3_cal1}
\begin{split}
    T^3(D^2 u, V ) = & \sum_{\alpha,\beta} \gamma^{\alpha \alpha} \eta^{\beta} T^3 (\phi^{\alpha}, \phi^{\alpha}, \phi^{\beta}) + \sum_{\alpha} \gamma^{\alpha \alpha} \eta^{n} T^3 (\phi^{\alpha}, \phi^{\alpha}, \phi^{n}) \\
    & + \sum_{\alpha,\beta} \sqrt{2} \gamma^{\alpha n} \eta^{\beta} T^3 (\phi^{\alpha}, \phi^{n}, \phi^{\beta}) + \sum_{\alpha} \sqrt{2} \gamma^{\alpha n} \eta^{n} T^3 (\phi^{\alpha}, \phi^{n}, \phi^{n}) \\
    & + \sum_{\beta} \gamma^{n n} \eta^{\beta} T^3 (\phi^{n}, \phi^{n}, \phi^{\beta}) + \gamma^{n n} \eta^{n} T^3 (\phi^{n}, \phi^{n}, \phi^{n}).   
\end{split}
\end{equation}
In order to simplify the upcoming calculations, we define the following quantities:
\begin{equation*}
\begin{split}
    & T^3_1 = \sum_{\alpha,\beta} \gamma^{\alpha \alpha} \eta^{\beta} T^3 (\phi^{\alpha}, \phi^{\alpha}, \phi^{\beta}), \quad  T^3_2 = \sum_{\alpha} \gamma^{\alpha \alpha} \eta^{n} T^3 (\phi^{\alpha}, \phi^{\alpha}, \phi^{n}), \\
    & T^3_3 = \sum_{\alpha}  \sqrt{2} \gamma^{\alpha n} \eta^{\alpha} T^3 (\phi^{\alpha}, \phi^{n}, \phi^{\alpha}), \quad T^3_4 = \sum_{\alpha \neq \beta}  \sqrt{2} \gamma^{\alpha n} \eta^{\beta} T^3 (\phi^{\alpha}, \phi^{n}, \phi^{\beta}), \\
    & T^3_5 = \sum_{\alpha} \sqrt{2} \gamma^{\alpha n} \eta^{n} T^3 (\phi^{\alpha}, \phi^{n}, \phi^{n}), \\
    & T^3_6 = \sum_{\alpha} \gamma^{n n} \eta^{\alpha} T^3 (\phi^{n}, \phi^{n}, \phi^{\alpha}), \quad  T^3_7 =  \gamma^{n n} \eta^{n} T^3 (\phi^{n}, \phi^{n}, \phi^{n}),    
\end{split}
\end{equation*}
which can be rewritten as
\begin{equation} \label{eq:T3_cal2}
\begin{split}
    & T^3_1 = \sum_{1\leq\alpha, \beta\leq n-1} \left(   \gamma^{\alpha \alpha} \gamma^{\beta \beta} \rho^{\beta} + \gamma^{\alpha \alpha} \frac{\gamma^{\beta n}}{\sqrt{2}} \rho^n \right)  T^3 (\phi^{\alpha}, \phi^{\alpha}, \phi^{\beta}), \\
    & T^3_2 = \sum_{1\leq\alpha, \beta\leq n-1} \gamma^{\alpha \alpha} \frac{\gamma^{\beta n}}{\sqrt{2}} \rho^{\beta} T^3 (\phi^{\alpha}, \phi^{\alpha}, \phi^{n}) + \sum_{1\leq\alpha\leq n-1} \gamma^{\alpha \alpha} \gamma^{n n} \rho^n   T^3 (\phi^{\alpha}, \phi^{\alpha}, \phi^{n}), \\
    & T^3_3 = \sum_{1\leq \alpha \leq n-1} \left(  \sqrt{2} \gamma^{\alpha n} \gamma^{\alpha \alpha} \rho^{\alpha} +  (\gamma^{\alpha n})^2  \rho^n  \right) T^3 (\phi^{\alpha}, \phi^{n}, \phi^{\alpha}),  \\ 
    & T^3_4 =  \sum_{1 \leq \alpha, \beta \leq n-1, \alpha \neq \beta} \left(  \sqrt{2} \gamma^{\alpha n} \gamma^{\beta \beta} \rho^{\beta} +  \gamma^{\alpha n} \gamma^{\beta n} \rho^n  \right) T^3 (\phi^{\alpha}, \phi^{n}, \phi^{\beta}), \\
    & T^3_5 =  \sum_{1\leq\alpha, \beta \leq n-1}  \gamma^{\alpha n} \gamma^{\beta n} \rho^{\beta} T^3 (\phi^{\alpha}, \phi^n, \phi^n) + \sum_{1\leq\alpha \leq n-1 } \sqrt{2} \gamma^{\alpha n} \gamma^{n n} \rho^n  T^3 (\phi^{\alpha}, \phi^n, \phi^n),   \\
    & T^3_6 = \sum_{1\leq\alpha \leq n-1} \left(  \gamma^{\alpha \alpha} \gamma^{n n} \rho^{\alpha} + \frac{\gamma^{\alpha n}}{\sqrt{2}} \gamma^{n n} \rho^n  \right) T^3 (\phi^{n}, \phi^{n}, \phi^{\alpha}),  \\
    & T^3_7 = \left( \sum_{1\leq\alpha \leq n-1}  \frac{\gamma^{n n} \gamma^{\alpha n}}{\sqrt{2}} \rho^{\alpha} + (\gamma^{n n})^2 \rho^n \right)  T^3 (\phi^{n}, \phi^{n}, \phi^{n}).
\end{split}
\end{equation}

With the above formulations, we are now ready to analyze the solutions of our boundary value problems.

\section{Contact angle problem} \label{sec:3}
In this section, we study the anisotropic mean curvature flow with prescribed contact angle boundary condition \eqref{eq:angpro}. Firstly, we establish a prior gradient estimate independent of time to the solution of \eqref{eq:angpro}. Then we show that we can apply the same technique to the elliptic version of the boundary problem the solution of which induces a translating solution to the parabolic problem. Finally, we study the asymptotic behavior of the solution to \eqref{eq:angpro}. For the isotropic case, similar approaches have been used in \cite{AW2}, \cite{MWW} and \cite{GMWW}. 

For convenience, we record here again our contact angle boundary value problem
\eqref{eq:angpro}: 
\begin{equation} \label{eq:angbry}
    \left\{ \begin{array}{cllcl}
            u_t & = & \sqrt{1+|Du|^2} \  D^2_{p_i p_j} F (Du,-1) u_{x_i x_j}   \quad &  \text{in} & \ {Q}_T , \\
            D_N u & = & -  \cos \theta \sqrt{1+|Du|^2}  \quad &  \text{on} & \ \Gamma_T ,  \\
         u(\cdot,0) & = & u_0(\cdot)  \quad \ & \text{in} & \ \Omega_0 .
    \end{array}  \right.
\end{equation}
Before we state the main results, we make the following assumptions for our
contact angle problem:
\begin{description}
\item[A1.] 
Let $\Omega\in \mathbb{R}^n$ be a smooth strictly convex bounded domain. There exists a positive constant $k_0$ such that the curvature matrix $K = \{ k_{ij} \}_{i,j=1}^{n-1}$ (differential of the normal map) at any point of $\partial\Omega$ satisfies the following condition,
\begin{equation*}
    K \geq k_0 I.
\end{equation*}
\item[A2.]
The contact angle function $\theta$ in \eqref{eq:angbry} can be extended to $\overline{\Omega}$ with $\theta \in C^3 (\overline{\Omega})$. Furthermore, there exist a positive constant, $\epsilon_1$ depending only on $\partial\Omega$ such that 
\begin{equation} \label{eq:ang_condA2}
    |\cos{\theta}|\leq \epsilon_1 <1, \quad \text{and} \quad \| D\theta\|_{C^1(\overline{\Omega})} \leq \epsilon_1.
\end{equation}
The first condition implies that $\theta$ is bounded away from $0$ and $\pi$.

\item[A3.] The initial condition $u_0$ is assumed to be $C^3(\overline{\Omega})$ and satisfies $D_N u_0 = - \cos \theta \sqrt{1+|Du_0|^2} $ on $\partial \Omega$.

\item[A4.] Recall $c_1$ and $c_2$ in Section \ref{sec:2.1} \eqref{eq:F_cond1}--\eqref{eq:F_cond3}. There exists a positive constant $\epsilon_2$ depending only on $n$ such that
\begin{equation} \label{eq:ang_condA4}
    \frac{c_2}{c_1} < \epsilon_2.
\end{equation}
\end{description}

\begin{theorem} \label{thm:3.1}
Assume {\bf A1} -- {\bf A4} and 
$\epsilon_1$, $\epsilon_2$ are sufficiently small, then
there exists a constant $C$ such that any solution $u(x,t) \in C^{3,2} (\overline{Q_T})$ of \eqref{eq:angbry}
satisfies
\begin{equation}\label{eq:ang_grad.est}
    \sup_{\overline{Q_T}} |Du| \leq C,
\end{equation}
where $C$ is independent of time.
\end{theorem}
With that, we can establish the long time behavior of the solution. The results will be stated in Section \ref{sec:longtimeangle}.

Before presenting the proof, we state the following lemma. The proof is omitted as it can be found in \cite{AW2} and also our previous paper \cite{CY1}.
\begin{lemma}\label{lem:3.2}
Let $u(x,t)$ be a smooth solution to \eqref{eq:angbry}. Then there exists a constant 
$C=C(u_0)$ such that 
\begin{equation*}
    \sup_{Q_T} |u_t|^2=\sup_{\Omega_0} |u_t|^2 \leq C.
\end{equation*}
\end{lemma} 

Now we proceed to the 

\noindent
\textbf{Proof of Theorem 3.1:} The idea of the proof is from \cite{GMWW}.
Firstly we define the following auxiliary function,
\begin{equation}
    \Psi(x) := \log W(x) + a_0 h(x),
\end{equation}
where 
\begin{equation}\label{W.CAC}
    W= v - \langle Du, Dh\rangle \cos \theta,
\end{equation}
$a_0$ is a positive number to be determined, and the function $h$ satisfies the following conditions with some positive constants $M_1$, $M_2$ and $k_1$ (see for example \cite[p.275]{CNS} for a construction of such a function),
\begin{equation}  \label{eq:h_cond}
\left\{ \begin{array}{cllcl}
           & h=0, \ D_N h=-1 \quad & \text{on} & \ \partial \Omega, \\
     & h<0, \ |Dh|\leq 1, \ D^2 h \geq k_1 I \quad & \text{in} & \ \Omega, \\
     & |D^2 h| \leq M_1,\ |D^3 h| \leq M_2 \quad & \text{in} & \ \Omega.
    \end{array}  \right.
\end{equation}
Notice that with the above condition on $h$, $W$ is a positive function so that
$\log W$ is well-defined. Furthermore, $\Psi$ as a function of $Du$, satisfies
 the following property: for $|Du|$ sufficiently large, there are constants $C_1 < C_2$ such that,
\[
C_1\log|Du|\leq \Psi \leq C_2\log|Du|.
\]

Now assume that $\Psi$ attains its maximum at $(x_0,t_0)\in \overline{Q_T}$.  The proof of this theorem considers the following two cases.

\subsection{Case 1: $(x_0,t_0) \in \Gamma_T$}
Note that in this case, only the
boundary condition in \eqref{eq:angbry} is needed. Hence the result is independent of the governing equation. See the same proof for the isotropic case in \cite{GMWW}. For reader's convenience, we outline it here.

At $(x_0,t_0)$, we choose geodesic coordinates $\{x_i \}_{i=1}^{n-1}$ 
for $\partial\Omega$ and $x_n $ to be along the inner normal direction. By means of parallel transport, we will extend the tangential vectors of $\partial\Omega$ at $x_0$ along the $\{x_i \}_{i=1}^{n-1}$ and (inward) $x_n $ directions. 
Under this coordinate system, at this point, we can write 
$D_Tu=(D_{T_1} u = u_{x_1}, \ldots, D_{T_{n-1}}u=u_{x_{n-1}})$ and $D_N u=u_{x_n}$ for the tangential and normal components of $Du$. Hence $|Du|^2 = |D_Tu|^2 + |u_{x_n}|^2 $.

Notice that from \eqref{eq:h_cond}, we have $ h_{x_i}=0, \ h_{x_n} = D_N h=-1$ on \ $\partial \Omega$, so that along $\partial\Omega$, it holds that
\begin{equation*}
    W = v + u_{x_n} \cos{\theta}.
\end{equation*}
Maximum principle suggests that at $(x_0,t_0)$,
\begin{equation} \label{eq:arg_max_bd1}
    \Psi_{x_i} =0, \quad 1\leq i \leq n-1,
\end{equation}
\begin{equation}  \label{eq:arg_max_bd2}
    \Psi_{x_n} \leq0.
\end{equation}
From \eqref{eq:arg_max_bd1} together with \eqref{eq:h_cond}, we have
\begin{equation*}
    0 = \frac{v_{x_i} + u_{x_n x_i} \cos{\theta} - u_{x_n} \theta_{x_i} \sin{\theta}}{W},
\end{equation*}
and hence
\begin{equation} \label{eq:arg_bd_vi}
    v_{x_i} = - u_{x_n x_i} \cos{\theta} + u_{x_n} \theta_{x_i} \sin{\theta}.
\end{equation}
Then differentiating the boundary condition of \eqref{eq:angbry} along the tangential directions, \eqref{eq:arg_bd_vi} gives,
\begin{equation*}
    u_{x_n x_i}  = - v_{x_i} \cos{\theta} + v \theta_{x_i} \sin{\theta} = u_{x_n x_i} \cos^2 \theta  - u_{x_n} \theta_{x_i} \cos{\theta} \sin{\theta} + v \theta_{x_i} \sin{\theta},
\end{equation*}
and we obtain
\begin{equation} \label{eq:arg_bd_uni}
    u_{x_n x_i} = v \theta_{x_i} (\cos{\theta} \cot{\theta} + \csc{\theta}).
\end{equation}

We also compute the following using the curvature matrix $K=\left\{k_{ij}\right\}$ of $\partial\Omega$:
\begin{equation} \label{eq:arg_bd_vn}
\begin{split}
    v_{x_n} & = \frac{\sum_{i=1}^{n-1} u_{x_i} u_{x_i x_n} }{v} - u_{x_n x_n} \cos{\theta} \\
    & = \frac{\sum_{i=1}^{n-1} \left( u_{x_i} u_{x_n x_i} + \sum_{j=1}^{n-1} k_{ij} u_{x_i} u_{x_j} \right) }{v} - u_{x_n x_n} \cos{\theta}.
\end{split}
\end{equation}
and
\begin{eqnarray} 
    W_{x_n} & = &v_{x_n} - \left( \sum_{k=1}^n u_{x_k} h_{x_k} \cos \theta \right)_{x_n} \nonumber\\
    & = &v_{x_n} - \sum_{k=1}^{n-1} u_{x_k x_n} h_{x_k} \cos \theta - u_{x_n x_n} h_{x_n} \cos \theta - \sum_{k=1}^n u_{x_k} h_{x_k x_n} \cos \theta
    \nonumber\\
    && + \sum_{k=1}^{n-1} u_{x_k} h_{x_k} \theta_{x_n} \sin{\theta} + u_{x_n} h_{x_n} \theta_{x_n} \sin{\theta}\nonumber \\
    & = &v_{x_n} +  u_{x_n x_n} \cos \theta - \sum_{k=1}^n u_{x_k} h_{x_k x_n} \cos \theta - u_{x_n} \theta_{x_n} \sin{\theta}.\label{eq:arg_bd_wn}
\end{eqnarray}
Now we apply \eqref{eq:arg_bd_vn} and \eqref{eq:arg_bd_wn} to \eqref{eq:arg_max_bd2}, then
\begin{equation} \label{eq:ang_bd_ineq}
    0 \geq \frac{\sum_{i=1}^{n-1} \left( u_{x_i} u_{x_n x_i} + \sum_{j=1}^{n-1} k_{ij} u_{x_i} u_{x_j} \right) }{vW} -  \frac{ \sum_{i=1}^n u_{x_i} h_{x_i x_n} \cos \theta  }{W} + \theta_{x_n} \cot{\theta} - a_0.
\end{equation}
Furthermore, applying \eqref{eq:arg_bd_uni} to \eqref{eq:ang_bd_ineq}, we have
\begin{eqnarray*}
    0 &\geq &  \frac{ \sum_{i,j=1}^{n-1} k_{ij} u_{x_i} u_{x_j} }{vW} + \frac{ (\cos{\theta} \cot{\theta} + \csc{\theta}) \sum_{i=1}^{n-1} u_{x_i} \theta_{x_i}  }{W} \\
    &&-  \frac{ \sum_{i=1}^{n-1} u_{x_i} h_{x_i x_n} + u_{x_n} h_{x_n x_n} \cos \theta  }{W} + \theta_{x_n} \cot{\theta} - a_0 \\
    &\geq & \frac{k_0 |D_Tu|^2}{1+ |D_Tu|^2 } - \frac{\epsilon_1 (M_1+3)}{1- \epsilon_1^2} - a_0,
\end{eqnarray*}
where we have made use of the fact that $vW = 1+|D_Tu|^2$ due to the boundary condition in \eqref{eq:angbry}. The above result implies that for any $\alpha \leq k_0/3$ with $\epsilon_1 \leq \frac{k_0}{9(M_1+3)}$,
\begin{equation*}
    |D_Tu|^2\leq 2.
\end{equation*}
By the boundary condition, we have 
$\displaystyle
|D_Tu|^2 = \sin^2\theta(1+|Du|^2)-1.
$
Hence an upper bound of $|Du|$ is obtained (as $\sin^2\theta \geq 1-\epsilon^2_1$). \\ \\

\subsection{Case 2: $(x_0,t_0)$ is an interior point of $\overline{Q_T}$} \label{sec:ang_grad_case_2}
In this case, the maximum principle implies at this point that
\begin{equation} \label{eq:ang_int_max1}
    \Psi_{x_i} = 0,
\end{equation}
\begin{equation}  \label{eq:ang_int_max2}
    \{\Psi_{x_i x_j}\} \leq 0,\ \Psi_t \geq 0.
\end{equation}
Without loss of generality, we can assume $|Du|>1$. Then the matrix $a^{ij}$ defined in \eqref{eq:def_a} is bounded. Since $\{a^{ij}\}$ is semi-positive definite, we also have from \eqref{eq:ang_int_max2} that 
\begin{equation} \label{eq:ang_int_ineq1}
    0 \geq a^{ij} \Psi_{x_i x_j} - \Psi_t.
\end{equation}
Now we calculate the quantities in the right hand side of \eqref{eq:ang_int_ineq1}:
\begin{eqnarray*}
       \Psi_{x_i} &=& \frac{W_{x_i}}{W} + a_0 h_{x_i}, \\ 
    \Psi_{x_i x_j} &=& \frac{W_{x_i x_j}}{W} - \frac{W_{x_i} W_{x_j}}{W^2} + a_0 h_{x_i x_j}, \\
    \Psi_t &=& \frac{W_t}{W},\\
    a^{ij} \Psi_{x_i x_j} - \Psi_t &=& \frac{a^{ij} W_{x_i x_j} -W_t}{W} - a_0^2 a^{ij} h_{x_i} h_{x_j} + a_0 a^{ij} h_{x_i x_j}\\
    &=& I_1 + I_2 + I_3,
\end{eqnarray*}
where
\begin{equation}
I_1 :=  \frac{a^{ij} W_{x_i x_j} -W_t}{W}, \quad
I_2 := -  a_0^2 \ a^{ij} h_{x_i} h_{x_j}, \quad
I_3 := a_0 \ a^{ij} h_{x_i x_j}.
\end{equation}
Hence \eqref{eq:ang_int_ineq1} gives
\begin{equation}\label{0I1I2I3}
0 \geq I_1 + I_2 + I_3 .
\end{equation}

To continue, we recall the notations introduced in Section 2.1. Then by the property of $h$, $I_2$ and $I_3$ can be bounded from below by:
\begin{equation}\label{angI2est}
    I_2 \geq - a_0^2 \big(\max_{\alpha} \tau^{\alpha \alpha} + \tau^{n n}\big),
\end{equation}
\begin{equation}\label{angI3est}
    I_3 \geq a_0 k_1 \big(\sum_{\alpha} \tau^{\alpha \alpha} + \tau^{n n}\big).
\end{equation}

Now we focus on $I_1$. 
For this, we calculate the following expressions.
\begin{eqnarray*}
     W_{x_i} & = & u_{x_k x_i} \left( \frac{u_{x_k}}{v} - h_{x_k} \cos \theta \right)  - u_{x_k} h_{x_k x_i} \cos \theta + u_{x_k} h_{x_k} \theta_{x_i} \sin \theta  \\
    & = &  S_k u_{x_k x_i} - u_{x_k} h_{x_k x_i} \cos \theta + u_{x_k} h_{x_k} \theta_{x_i} \sin \theta , \\
    W_t & = &  u_{x_k,t} \left( \frac{u_{x_k}}{v} - h_{x_k} \cos \theta \right) = S_k u_{x_k,t}, \\
    W_{x_i x_j} & = &  v_{x_i x_j} - \left( u_{x_k} h_{x_k} \cos \theta \right)_{x_i x_j} \\
    & = &  \frac{u_{x_k x_i} u_{x_k x_j}}{v} - \frac{u_{x_k} u_{x_k x_i} u_{x_l} u_{x_l x_j}}{v^3} + \frac{u_{x_k} u_{x_k x_i x_j}}{v}  - u_{x_i x_j x_k} h_{x_k} \cos \theta \\
    & & - 2 u_{x_k x_j} h_{x_k x_i} \cos \theta - u_{x_k} h_{x_k x_i x_j} \cos \theta - 2 (u_{x_k} h_{x_k})_{x_i} (\cos \theta)_{x_j} - u_{x_k} h_{x_k} (\cos \theta)_{x_i x_j}.
\end{eqnarray*}
By differentiating the governing equation in \eqref{eq:angbry}, we obtain
\begin{equation*}
    u_{x_k,t} = (a^{ij})_{x_k} u_{x_i x_j} + a^{ij} u_{x_i x_j x_k}.
\end{equation*}
Then we compute $a^{ij} W_{x_i x_j} - W_t$:
\begin{equation}
\label{eq:ang_I1}
    a^{ij} W_{x_i x_j} - W_t
    = J_1 + J_2 + J_3 + J_4,
\end{equation}
where
\begin{eqnarray*}
    J_1 &  := &   \frac{1}{v} a^{ij} u_{x_j x_l} \left( \delta_{kl} - \frac{u_{x_k} u_{x_l}}{v^2} \right) u_{x_k x_i}, \\
    J_2 & := &  - (a^{ij})_{x_k} u_{x_i x_j} \left(\frac{u_{x_k}}{v} - h_{x_k} \cos \theta\right), \\
    J_3 & := & - a^{ij} u_{x_k} \big(h_{x_i x_j x_k} \cos \theta - 2 h_{x_k x_i} (\sin \theta )\theta_{x_j} - h_{x_k} (\cos \theta )\theta_{x_i} \theta_{x_j} - h_{x_k} (\sin\theta) \theta_{x_i x_j}\big), \\
    J_4 & := &  - 2 a^{ij} u_{x_k x_i} \big(h_{x_k x_j} \cos \theta - h_{x_k} (\sin \theta) \theta_{x_j}\big).
\end{eqnarray*}
Hence
\[
I_1 = \frac{1}{W}(J_1 + J_2 + J_3 + J_4). 
\]
We will analyze  $J_i$'s in the order of: $J_3, J_4, J_1$ and $J_2$. 

For $J_3$, it can simply be estimated by
\begin{equation} \label{eq:ang_J3_est}
    J_3 \geq - C \Big(|\cos{\theta}| + |D\theta| + |D^2 \theta|\Big) |Du|,
\end{equation}
where $C$ depends on the upper bound of $a^{ij}$ and $h$.

Next we apply the orthonormal basis we have chosen in Section \ref{Sec:2.2} for 
$\mathbb{R}^n$ and $n\times n$ symmetric matrices. To simplify notations, we 
introduce
\begin{equation*}
    A := \{ a^{ij} \}, \quad G = \{ g^{ij} \} := \left\{ \delta_{ij} - \frac{u_{x_i} u_{x_j}}{v^2} \right\}, \,\,\,
    H := \Big\{ h_{x_i x_j} \cos \theta - h_{x_i} (\sin \theta)\theta_{x_j} \Big\}.
\end{equation*}
The following are the computational formulas for $G$ and $A$ on basis elements: 
\begin{eqnarray*}
     G \Phi^{\alpha \alpha} & = & \Phi^{\alpha \alpha}, \\
     G \Phi^{\alpha \beta} & = & \Phi^{\alpha \beta}, \\
     G \Phi^{\alpha n} & = & \frac{1}{\sqrt{2}} \left( \phi^{\alpha} (\phi^{n})^T + \frac{1}{v^2} \phi^{n} (\phi^{\alpha})^T \right), \\
     G \Phi^{n n} & = & \frac{1}{v^2} \Phi^{n n}, \\
     \text{tr}(A \Phi^{\alpha \alpha}) & = & \tau^{\alpha \alpha}, \\
     \text{tr}(A \Phi^{\alpha \beta}) & = & \sqrt{2} \tau^{\alpha \beta}, \\
     \text{tr}(A \Phi^{\alpha n}) & = & \sqrt{2} \tau^{\alpha n}, \\
     \text{tr}(A \Phi^{n n}) & = & \tau^{n n}. 
\end{eqnarray*}

For $J_4$, we write it as,
\begin{equation*}
    J_4 = - 2 a^{ij} u_{x_k x_i} \left(h_{x_k x_j} \cos \theta - h_{x_k} (\sin \theta )\theta_{x_j} \right) = - 2 \ \text{tr}(A \ D^2 u \ H ).
\end{equation*}
Notice that $\displaystyle D\Psi = \frac{a_0W\ Dh - HDu + D^2u \ S}{W}$
where
\begin{equation*}
     S:= \frac{Du}{v} - Dh\cos\theta,\,\,\,\text{i.e.}\,\,\,
     S_k  := \frac{u_{x_k}}{v} - h_{x_k} \cos \theta.
\end{equation*}
Hence \eqref{eq:ang_int_max1} $D\Psi = 0$ can be written as
\begin{equation}
    D^2u \  S = -a_0 W \ Dh + H \ Du,
\end{equation}
which implies 
\begin{equation} \label{eq:ang_int_max1_r}
\begin{split}
    &  (\phi^{\alpha})^T \  D^2 u \ S = -a_0 W \ (\phi^{\alpha})^T \ Dh + (\phi^{\alpha})^T \ H \ Du, \\ 
    &  (\phi^n)^T \  D^2 u \ S = -a_0 W \ (\phi^n)^T \ Dh + (\phi^n)^T \ H \ Du.
\end{split}
\end{equation}

To continue the estimate on $J_4$, we define the following 
\begin{equation} \label{eq:ang_rho_def}
\begin{split}
    & \rho^{\alpha} := (\phi^{\alpha})^T \ S = -\big((\phi^\alpha)^TDh\big)\cos\theta \,\,\, \\
    &\rho^n := (\phi^n)^T \ S = 
    \frac{|Du|}{v} - \frac{(Du)^TDh}{|Du|}\cos\theta, 
\end{split}
\end{equation}
\begin{equation}  \label{eq:ang_eta_def}
\begin{split}
    & \eta^{\alpha} :=  (\phi^{\alpha})^T \  D^2 u \ S = \gamma^{\alpha \alpha} \rho^{\alpha} + \frac{\gamma^{\alpha n}}{\sqrt{2}} \rho^n, \\
    & \eta^n := (\phi^n)^T \  D^2 u \ S = \sum_{\alpha}  \frac{\gamma^{\alpha n}}{\sqrt{2}} \rho^{\alpha} + \gamma^{n n} \rho^n.
\end{split}
\end{equation}
Therefore
\begin{equation} \label{eq:ang_D2u_S}
    D^2 u \ S = \sum_{\alpha} \eta^{\alpha} \phi^{\alpha} + \eta^n \phi^n.
\end{equation}
Since $|Du| \geq 1$, we have:
\begin{equation}
    \left|\rho^{\alpha} \right| \leq | \cos{\theta}| \quad \text{and} \quad \frac{1}{4} \leq \rho^n \leq 2.
\end{equation}
Note that \eqref{eq:ang_int_max1_r}, \eqref{eq:ang_eta_def} and \eqref{eq:ang_D2u_S} suggest a way to express $\gamma^{\alpha n},\ \gamma^{n n}$ by $\gamma^{\alpha \alpha}$:
\begin{equation*}
\begin{split}
    & \gamma^{\alpha n} = - \frac{-\sqrt{2} \rho^{\alpha}}{\rho^n} \gamma^{\alpha \alpha} + O(1) a_0 W + O\big(|\cos{\theta} + |D\theta||\big) |Du|, \\
    & \gamma^{n n} = \sum_{\alpha} \frac{(\rho^{\alpha})^2}{(\rho^n)^2} \gamma^{\alpha \alpha} + O(1) a_0 W + O\big(|\cos{\theta} + |D\theta||\big) |Du|.
\end{split}
\end{equation*}
Then 
\begin{equation} \label{eq:ang_J4_est}
\begin{split}
    J_4 \geq & -2 \| H \|_{L^{\infty}} \left( \sum_{\alpha} \gamma^{\alpha \alpha} \text{tr}(A\Phi^{\alpha \alpha}) +  \sum_{\alpha} \gamma^{\alpha n} \text{tr}(A\Phi^{\alpha n}) + \gamma^{n n} \text{tr}(A\Phi^{n n}) \right) \\
    \geq & - C \left( \sum_{\alpha} \gamma^{\alpha \alpha} \tau^{\alpha \alpha} +  \sum_{\alpha} \gamma^{\alpha n} \sqrt{2} \tau^{\alpha n} + \gamma^{n n} \tau^{n n}  \right) \\
    = & - C \left(  \sum_{\alpha} \gamma^{\alpha \alpha} \left( \tau^{\alpha \alpha} - \frac{2\rho^{\alpha}}{\rho^n} \tau^{\alpha n} + \frac{(\rho^{\alpha})^2}{(\rho^n)^2} \tau^{n n} \right)\right. \\
    & \quad\quad\quad\left.+ \Big(O(1) a_0 W + O\big(|\cos{\theta} + |D\theta||\big) |Du| \Big) \Big(\sum_{\alpha} \sqrt{2} \tau^{\alpha n} + \tau^{n n}\Big)  \right) .
\end{split}
\end{equation} 

For $J_1$, we write
\begin{equation}
    J_1 = \frac{1}{v} \text{tr}(A\ D^2 u\ G\ D^2 u).
\end{equation}
Consider the bilinear form on $n \times n$ symmetric matrices,
\begin{equation*}
    \mathcal{A}(\cdot, \cdot) : = \text{tr}(A \ \cdot \  G\ \cdot \ ).
\end{equation*}
Evaluating $\mathcal{A}$ on the basis vectors $\Phi^{\alpha\beta}$ gives:
\begin{eqnarray*}
         \mathcal{A} (\Phi^{\alpha \alpha}, \Phi^{\alpha \alpha} ) & = &  \tau ^{\alpha \alpha}, \\
         \mathcal{A} (\Phi^{\alpha n}, \Phi^{\alpha n}) & = & \frac{1}{2} (\frac{1}{v^2} \tau ^{\alpha \alpha} + \tau^{nn} ), \\
         \mathcal{A} (\Phi^{n n}, \Phi^{n n}) & = & \frac{1}{v^2} \tau^{n n},\\
     \mathcal{A} (\Phi^{n n}, \Phi^{\alpha n} ) & = & \frac{1}{\sqrt{2}} \frac{1}{v^2} \tau^{\alpha n}, \\
     \mathcal{A} (\Phi^{\alpha n}, \Phi^{\beta n} ) & = & \frac{1}{2} \frac{1}{v^2} \tau^{\alpha \beta}, \\
     \mathcal{A} (\Phi^{\alpha \alpha}, \Phi^{\alpha n} ) & = & \frac{1}{\sqrt{2}} \tau^{\alpha n}, \\
\end{eqnarray*}
and $\mathcal{A}$ vanishes in all the remaining cases. Now
\begin{equation} \label{eq:ang_J1_est}
\begin{split}
    J_1  = & \ \frac{1}{v} \mathcal{A}(D^2 u, D^2 u) \\
     =  & \ \frac{1}{v} \Bigg( \sum_{\alpha} (\gamma^{\alpha \alpha})^2 \tau^{\alpha \alpha} +  \sum_{\alpha} \frac{1}{2}(\gamma^{\alpha n})^2  (\frac{1}{v^2} \tau^{\alpha \alpha} + \tau^{nn}) + \frac{1}{v^2} (\gamma^{n n})^2 \tau^{n n} \\
    & \ \sum_{\alpha} \sqrt{2} \gamma^{\alpha \alpha} \gamma^{\alpha n} \tau^{\alpha n} + \sum_{\alpha \neq \beta} \frac{1}{2} \frac{1}{v^2} \gamma^{\alpha n} \gamma^{\beta n} \tau^{\alpha \beta} + \sum_{\alpha}  \sqrt{2} \gamma^{\alpha n} \gamma^{n n} \tau^{\alpha n} \Bigg).
\end{split}
\end{equation}

Finally we consider $J_2$. Recall the tensor $T^3$ defined in Section \ref{Sec:2.2},
\begin{equation}
\begin{split}
    J_2 & = - (a^{ij})_{x_k} u_{x_i x_j} \left(\frac{u_{x_k}}{v} - h_{x_k} \cos \theta \right) \\
    & = -D_{p_l} \left(|p| D^2_{p_i p_j} F \right)\Big|_{(Du,-1)} u_{x_i x_j} u_{x_k x_l} \left(\frac{u_{x_k}}{v} - h_{x_k} \cos \theta \right) \\
    & = T^3 (D^2 u , D^2 u\,S).
\end{split}
\end{equation}
Then we can apply the result \eqref{eq:T3_cal1} and \eqref{eq:T3_cal2} with $V = D^2 u\,S$.

We claim that Theorem \ref{thm:3.1} will be proved if  following statement is true:
\begin{equation}\label{eq:ang_I1_res}
J_1 + J_2 + J_3 + J_4 \geq -C(\theta) v.  
\end{equation}
To see this, note that
\begin{equation*}
    I_1 = \frac{J_1 + J_2 + J_3 + J_4}{W} \geq -C(\theta).
\end{equation*}
Combining the above with the estimates \eqref{angI2est} and \eqref{angI3est} for $I_2$ and $I_3$, we obtain by \eqref{0I1I2I3} that
\begin{equation*}
    0 \geq I_1 + I_2 + I_3 \geq -C(\theta) - a_0^2 \left(\max_{\alpha} \tau^{\alpha \alpha} + \tau^{n n} \right) + a_0 k_1 \left(\sum_{\alpha} \tau^{\alpha \alpha} + \tau^{n n} \right).
\end{equation*}
Recall that $\tau^{\alpha \alpha}$ and $\tau^{n n}$ satisfy \eqref{eq:F_cond1}. By taking $a_0 =\min \{\frac{k_1}{2},  \frac{k_0}{3}, 1 \}$ and $\epsilon_1$ sufficiently small, we have
\begin{equation*}
    v \leq C,
\end{equation*}
where $C$ is independent of $T$ leading to the main statement \eqref{eq:ang_grad.est} of the Theorem.

\subsection{Proof of \eqref{eq:ang_I1_res}}
We first decompose its left hand side as
\begin{eqnarray}
J_1 + J_2 + J_3 + J_4 
&=& 
\frac{1}{10} \frac{1}{v}  \sum_{\alpha} (\gamma^{\alpha \alpha})^2 \tau^{\alpha \alpha} + J_3 + J_4 \\
&& 
+ J_1 - \frac{1}{10} \frac{1}{v}  \sum_{\alpha} (\gamma^{\alpha \alpha})^2 \tau^{\alpha \alpha} +J_2,
\end{eqnarray}
We claim that 
     \begin{equation} \label{eq:ang_cl1}
         \frac{1}{10} \frac{1}{v}  \sum_{\alpha} (\gamma^{\alpha \alpha})^2 \tau^{\alpha \alpha} + J_3 + J_4 \geq -Cv.
     \end{equation}
and

     \begin{equation} \label{eq:ang_cl2}
    J_1 - \frac{1}{10} \frac{1}{v}  \sum_{\alpha} (\gamma^{\alpha \alpha})^2 \tau^{\alpha \alpha} +J_2 \geq 0.
   \end{equation}
   
Recall that from \eqref{eq:ang_J3_est} and \eqref{eq:ang_J4_est}, $J_3$ and $J_4$ are bounded 
from below by a linear form of $\gamma^{\alpha \alpha}$ for $1\leq \alpha \leq n-1$. Then \eqref{eq:ang_cl1} follows naturally by the fact that $\sum_{\alpha} (\gamma^{\alpha \alpha})^2 \tau^{\alpha \alpha}$ is a positive quadratic form on $\gamma^{\alpha \alpha}$.

For the claim \eqref{eq:ang_cl2}, we analyze its left hand side by 
means of \eqref{eq:ang_J1_est}. Specifically, we write it as a bilinear form  of
$\gamma = (\gamma^{11}, \cdots, \gamma^{n-1,n-1}, \gamma^{1n}, \cdots, \gamma^{n-1,n}, \gamma^{nn}  )^T$. To this end, we have
\begin{equation}
\begin{split}
  \,\,\,&  J_1 - \frac{1}{10} \frac{1}{v}  \sum_{\alpha} (\gamma^{\alpha \alpha})^2 \tau^{\alpha \alpha} +J_2 \\
  = \,\,\,&  \sum_{\alpha} \left( \frac{9}{10v}  \tau^{\alpha \alpha} + \rho^{\alpha} T^3 (\phi^{\alpha}, \phi^{\alpha}, \phi^{\alpha}) \right)   (\gamma^{\alpha \alpha})^2  \\
  & + \sum_{\alpha \neq \beta}  \rho^{\beta} T^3 (\phi^{\alpha}, \phi^{\alpha}, \phi^{\beta})    \gamma^{\alpha \alpha} \gamma^{\beta \beta} \\
  & + \sum_{\alpha}    \left( \rho^n T^3 (\phi^{\alpha}, \phi^{n}, \phi^{\alpha}) + \rho^{\alpha} T^3 (\phi^{\alpha}, \phi^n, \phi^n) + \frac{1}{2v^3} \tau^{\alpha \alpha} +  \frac{1}{2v} \tau^{nn} \right) (\gamma^{\alpha n})^2 \\
  & + \sum_{\alpha \neq \beta}  \left( \rho^n  T^3 (\phi^{\alpha}, \phi^{n}, \phi^{\beta}) + \rho^{\beta} T^3 (\phi^{\alpha}, \phi^n, \phi^n) + \frac{1}{2v^3}  \tau^{\alpha \beta}  \right)  \gamma^{\alpha n} \gamma^{\beta n} \\
  & + \sum_{\alpha}  \left( \frac{\sqrt{2}}{v}  \tau^{\alpha n}  + \frac{\rho^n}{\sqrt{2}}  T^3 (\phi^{\alpha}, \phi^{\alpha}, \phi^{\alpha}) + \frac{\rho^{\alpha}}{\sqrt{2}} T^3 (\phi^{\alpha}, \phi^{\alpha}, \phi^{n}) + \sqrt{2} \rho^{\alpha}  T^3 (\phi^{\alpha}, \phi^{n}, \phi^{\alpha}) \right) \gamma^{\alpha \alpha} \gamma^{\alpha n} \\
  & + \sum_{\alpha \neq \beta}  \left( \frac{\rho^n}{\sqrt{2}}  T^3 (\phi^{\alpha}, \phi^{\alpha}, \phi^{\beta}) + \frac{\rho^{\beta}}{\sqrt{2}} T^3 (\phi^{\alpha}, \phi^{\alpha}, \phi^{n}) +  \sqrt{2} \rho^{\alpha}  T^3 (\phi^{\beta}, \phi^{n}, \phi^{\alpha})  \right) \gamma^{\alpha \alpha} \gamma^{\beta n} \\
  & + \left( \rho^n  T^3 (\phi^{n}, \phi^{n}, \phi^{n})  + \frac{1}{v^3}  \tau^{n n}  \right)  (\gamma^{n n})^2 \\
  & + \sum_{\alpha} \left( \sqrt{2}  \rho^n  T^3 (\phi^{\alpha}, \phi^n, \phi^n) + \frac{\rho^n}{\sqrt{2}} T^3 (\phi^{n}, \phi^{n}, \phi^{\alpha}) +  \frac{\rho^{\alpha}}{\sqrt{2}} T^3 (\phi^{n}, \phi^{n}, \phi^{n}) + \sqrt{2} \tau^{\alpha n}  \right) \gamma^{\alpha n} \gamma^{n n}  \\
  & + \sum_{\alpha} \left(  \rho^n  T^3 (\phi^{\alpha}, \phi^{\alpha}, \phi^{n}) + \rho^{\alpha} T^3 (\phi^{n}, \phi^{n}, \phi^{\alpha})   \right) \gamma^{\alpha \alpha} \gamma^{n n} \\
  :=\,\,\, & \mathcal{B}(\gamma, \gamma),
\end{split}
\end{equation}
where the symmetric bilinear form $\mathcal{B}$ has the following matrix representation
(recall that $1\leq\alpha, \beta \leq n-1$):
\begin{equation} \label{eq:B_matrep}
\mathcal{B} = \left(
\begin{array}{cc|cc|c}
B^{\alpha\alpha, \alpha\alpha} &  B^{\alpha\alpha,\beta\beta}  &  B^{\alpha\alpha,\alpha n}  & B^{\alpha\alpha,\beta n} & B^{\alpha\alpha,n n}  \\
B^{\beta\beta,\alpha\alpha} &  B^{\beta\beta,\beta\beta}  & B^{\beta \beta,\alpha n} & B^{\beta \beta,\beta n} & B^{\beta \beta, n n} \\
\hline
 B^{\alpha n, \alpha\alpha} & B^{\alpha n,\beta\beta} &   B^{\alpha n, \alpha n}  &   B^{\alpha n,\beta n}   &  B^{\alpha n,nn}   \\
 B^{\beta n, \alpha \alpha}& B^{\beta n,\beta\beta}  &   B^{\beta n,\alpha n}   &  B^{\beta n,\beta n}  &  B^{\beta n,nn}   \\
\hline
 B^{nn, \alpha\alpha} & B^{nn,\beta\beta}   & B^{nn,\alpha n}   &  B^{n n, \beta n}  & B^{nn,nn}
\end{array}
\right).
\end{equation}
We only need to check that $\mathcal{B}(\cdot,\cdot)$ is semi-positive definite when $\epsilon_1,\epsilon_2$ are sufficiently small. 

Note that the entries of $\mathcal{B}$ depend on $\tau^{(\cdot, \cdot)}$, $T^3(\cdot, \cdot, \cdot)$ and $\rho^{(\cdot)}$. They consist of terms that can be bounded from above or below. This is where we crucially make use of the estimates in Section \ref{sec:2.1}, in particular Lemma \ref{lem:2.4} for $T^3(\cdot, \cdot, \cdot)$, \eqref{eq:F_cond1}--\eqref{eq:F_cond3} for $\tau^{(\cdot, \cdot)}$ and \eqref{eq:ang_rho_def} for $\rho^{(\cdot)}$. To be precise, 
\begin{enumerate}
\item 
the \emph{diagonal} entries of $\mathcal{B}$ consist of only the following types of terms:
\[
B^{\alpha\alpha,\alpha\alpha}\,\,(B^{\beta\beta,\beta\beta}), 
\ B^{\alpha n,\alpha n}\,\,(B^{\beta n,\beta n}), 
\ B^{nn,nn}.
\]
They are dominated by the following terms shown with their respective lower bounds:
\begin{eqnarray*}
\tau^{\alpha\alpha} \geq c_1,\,\,\,
T^3(\phi^\alpha,\phi^n,\phi^\alpha)\geq \frac{c_1}{v},\,\,\,
T^3(\phi^n,\phi^n,\phi^n)\geq \frac{c_1}{v^3},\,\,\,
\rho^n \geq \frac14.
\end{eqnarray*}
From the above, we infer that there exists a positive constant $c_1'=c_1'(\theta,c_1,c_2)$ such that
\begin{equation*}
    B^{\alpha\alpha,\alpha\alpha}, \ B^{\alpha n,\alpha n} \geq \frac{c_1'}{v}, \quad B^{nn,nn} \geq \frac{c_1'}{v^3}.
\end{equation*}
\item 
the \emph{non-diagonal} entries of $\mathcal{B}$ consist of the following type of terms:
\[
    B^{\alpha\alpha,\beta\beta}, \ B^{\alpha \alpha, \alpha n} , \ B^{\alpha \alpha,\beta n}, \  B^{\alpha n,\beta n}, \quad \text{and}\quad
    B^{\alpha \alpha, n n} , \  B^{\alpha n, n n}.
\]
They are bounded from above by the following terms shown again with their respective
upper bounds:
\[
|\tau^{\alpha\beta}|\leq c_2,\,\,\,
|\tau^{\alpha n}|\leq\frac{c_2}{v^2}
\]\[
\left|T^3(\phi^\alpha, \phi^\alpha, \phi^\beta)\right|,\,\,\,
\left|T^3(\phi^\alpha, \phi^n, \phi^\beta)\right|\,\,\leq \frac{c_2}{v},
\]\[
\left|T^3(\phi^\alpha, \phi^\alpha, \phi^n)\right|,\,\,\,
\left|T^3(\phi^\alpha, \phi^n, \phi^n)\right|,\,\,\,
\left|T^3(\phi^n, \phi^n, \phi^\alpha)\right|\,\,
\leq \frac{c_2}{v^3},
\]\[
\left|\rho^\alpha\right| \leq |\cos\theta|.
\]
Then we can similarly infer that there exists a positive constant $c_2'=c_2'(\theta,c_1,c_2)$ such that
\[
B^{\alpha\alpha,\beta\beta}, \ B^{\alpha \alpha, \alpha n} , \ B^{\alpha \alpha,\beta n}, \  B^{\alpha n,\beta n} \leq \frac{c_2'}{v}, 
\quad \text{and}\quad
    B^{\alpha \alpha, n n} , \  B^{\alpha n, n n} \leq \frac{c_2'}{v^3}
    \,\,\,\text{for $\alpha\neq\beta$.}
\]
\end{enumerate}
Hence we conclude that for $\epsilon_1$ small with $|\cos\theta|\leq \epsilon_1$
and $\epsilon_2=\frac{c_2}{c_1}$ small, the bilinear form $\mathcal{B}(\cdot,\cdot)$ is semi-positive definite. Therefore, \eqref{eq:ang_cl2} holds. 

The above completes the proof of the Theorem.
\hfill$\square$

\subsection{Elliptic problem and long time behavior}\label{sec:longtimeangle}
In order to analyze the long time behavior of \eqref{eq:angbry},  we consider the elliptic version of the boundary value problem \eqref{eq:elang_pro} as introduced in Section \ref{sec:1},
\begin{equation} \label{eq:ang_el_pro}
\left\{ \begin{array}{cllcl}
     \lambda & = & \sqrt{1+|Dw|^2} \ D^2_{p_i p_j} F (Dw,-1) w_{x_i x_j}    
     \quad & \text{in} & \ \Omega,  \\
     D_N w & = & - \sqrt{1+|Dw|^2} \cos \theta \quad & \text{on} & \ \partial \Omega ,
\end{array} \right.
\end{equation}
and prove that there exists a unique $\lambda$ corresponding to $\theta$ such that a solution to \eqref{eq:ang_el_pro} exists and it is unique. This solution naturally leads to a translating solution to \eqref{eq:angbry} given by 
$\Tilde{w}(x,t) = w(x) + \lambda t$. Furthermore, we show that any solution to \eqref{eq:angbry} converges, up to a constant, to this translating solution as $t \rightarrow \infty$.
\begin{theorem} \label{thm:3.3}
    Assume {\bf A1} -- {\bf A4} and 
$\epsilon_1$, $\epsilon_2$ are sufficiently small. Then there exists a unique $\lambda$ and a corresponding smooth solution $w$ for \eqref{eq:ang_el_pro}. The function $w$ is unique up to a constant. 
\end{theorem}

\begin{proof}
We consider the following eigenvalue problem firstly: 
\begin{equation} \label{eq:ang_el_eigpro}
\left\{ \begin{array}{cllcl}
     \epsilon w^{\epsilon} & = & \sqrt{1+|Dw^{\epsilon}|^2} \ D^2_{p_i p_j} F (Dw^{\epsilon},-1) w^{\epsilon}_{x_i x_j}    \quad & \text{in} & \ \Omega,  \\
     D_N w^{\epsilon} & = & - \sqrt{1+|Dw^{\epsilon}|^2} \cos \theta \quad & \text{on} & \ \partial \Omega.
\end{array} \right.
\end{equation}
A barrier argument as in \cite{AW2} shows $\epsilon w^{\epsilon}$ is uniformly bounded -- see also the proof in \cite{GMWW}. We demonstrate below that a uniform gradient estimate to $D w^{\epsilon}$ can be established following the same steps as in Theorem \ref{thm:3.1}.  

Define as before the function $\Psi = \log W + a_0h$
where  $W = \sqrt{1+|Dw^{\epsilon}|^2} - w^{\epsilon}_{x_k} h_{x_k} \cos \theta $. Then similar to Theorem 3.1, the maximum of $\Psi(x)$ can occur on the boundary or in the interior. The boundary case is identical to {\bf Case 1} in the proof of Theorem 3.1. The interior case is also similar as outlined below.
At the maximum point, the maximum principle suggests 
\begin{equation*}
    0 \geq a^{ij} \Psi_{x_i x_j} = \frac{a^{ij} W_{x_i x_j}}{W}  - a_0^2 a^{ij} h_{x_i} h_{x_j} + a_0 a^{ij} h_{x_i x_j}.
\end{equation*}
We also claim 
\begin{equation*}
    \frac{a^{ij} W_{x_i x_j}}{W} \geq -C.
\end{equation*}
This is shown by differentiating the governing equation in \eqref{eq:ang_el_pro}
and substituting the third order derivative of $w^{\epsilon}$ in $\Psi_{x_i x_j}$
so that we can arrive at
\begin{equation*}
    \frac{a^{ij} W_{x_i x_j}}{W} = \frac{1}{W} (J_1 + J_2' + J_3 + J_4)
\end{equation*}
where $J_1, J_3, J_4$ are exactly the same as in \eqref{eq:ang_I1} by substituting $u$ with $w^{\epsilon}$, but 
\begin{equation*}
    J_2' = J_2 + \epsilon (Dw^{\epsilon})^T S
    =J_2 + \epsilon (Dw^{\epsilon})^T\left(\frac{Dw^\epsilon}{\sqrt{1+|Dw^\epsilon|^2}}-Dh\cos\theta\right)\geq J_2,
\end{equation*}
by the assumption that $|Dw^{\epsilon}| \geq 1$ and $\cos\theta$ is small. With these, we can apply the same process as in the interior case to obtain the gradient estimate of $w^{\epsilon}$, $|Dw^\epsilon| \leq C$.

The uniform gradient estimate implies $|D(\epsilon w^{\epsilon})| \rightarrow 0$ 
and hence $\epsilon w^{\epsilon}$ converges to some constant $\lambda$. 
The uniqueness of $\lambda$ follows from maximum principle and the Hopf Lemma as in \cite{AW1}.
\end{proof}

From Theorem \ref{thm:3.3}, for a solution $w(x)$ to the elliptic boundary problem \eqref{eq:ang_el_pro}, clearly $\Tilde{w}(x,t) = w(x) + \lambda t$ solves the following parabolic boundary problem:
\begin{equation} \label{eq:ang_transol}
    \left\{ \begin{array}{cllcl}
            u_t & = & \sqrt{1+|Du|^2} \  D^2_{p_i p_j} F (Du,-1) u_{x_i x_j}   \quad &  \text{in} & \ {Q}_T , \\
            D_N u & = & -  \cos \theta \sqrt{1+|Du|^2}  \quad &  \text{on} & \ \Gamma_T ,  \\
         u(\cdot,0) & = & w(\cdot)  \quad \ & \text{in} & \ \Omega_0.
    \end{array}  \right.
\end{equation}
The only difference to \eqref{eq:angbry} is the initial data.

Now we state our final results describing the asymptotic behavior of our solution $u$.  
Their proofs are basically the same as those in \cite{AW2}, \cite{GMWW}, and 
our parallel work \cite{CY1}. Hence we omit their proofs.
\begin{corollary} \label{cor:3.4}
    For any solution $u(x,t)$ to \eqref{eq:angbry}, there exists a positive constant $C$ independent of time, such that
    \begin{equation*}
        |u(x,t)-\lambda t| \leq C.
    \end{equation*}
\end{corollary}

\begin{theorem} 
    Let $u$ be any solution of \eqref{eq:angbry}. Then
    $u(x,t)-(w(x)+\lambda t)$ converges to a constant as $t\longrightarrow+\infty$, where $w$ is a solution of \eqref{eq:ang_el_pro}.
\end{theorem}


\section{Neumann problem} \label{sec:4}
In this section, we study the Neumann boundary problem \eqref{eq:neupro} as introduced in Section \ref{sec:1}:
\begin{equation} \label{eq:Neupro}
    \left\{ \begin{array}{cllcl}
            u_t & = & \sqrt{1+|Du|^2} \  D^2_{p_i p_j} F (Du,-1) u_{x_i x_j}   \quad &  \text{in} & \ {Q}_T , \\
            D_N u & = & \varphi  \quad &  \text{on} & \ \Gamma_T ,  \\
         u(\cdot,0) & = & u_0(\cdot)  \quad \ & \text{in} & \  \Omega_0 .
    \end{array}  \right.
\end{equation}
 Results in this section are organized in a similar way to the previous contact angle problem. The main result is a gradient estimate independent of time. Then we apply the same technique to the elliptic Neumann problem and derive the long time convergence result.

 Similar to the contact angle boundary condition, we make the following assumptions throughout this section:

\begin{description}
\item[B1.] 
Let $\Omega\in \mathbb{R}^n$ be a smooth strictly convex bounded domain. There exists a positive constant $k_0$ such that the curvature matrix $\{ k_{ij} \}_{i,j=1}^{n-1}$ at any point of $\partial\Omega$ satisfies the following condition,
\begin{equation*}
    \{ k_{ij} \} \geq k_0 \{ \delta_{ij} \}.
\end{equation*}
\item[B2.]
The Neumann boundary condition $\varphi$ in \eqref{eq:Neupro} can be extended to $\overline{\Omega}$ with $\varphi \in C^3 (\overline{\Omega})$. Furthermore, there exist a positive constant $\epsilon_1$ such that 
\begin{equation} \label{eq:neu_condA2}
    \|\varphi\|_{C^0(\overline{\Omega})},\,\,\,
    \|D\varphi\|_{C^0(\overline{\Omega})},\,\,\,
    \|D^2\varphi\|_{C^0(\overline{\Omega})} <\epsilon_1.
\end{equation} 

\item[B3.] The initial condition $u_0$ is assumed to be $C^3(\overline{\Omega})$ and satisfies $D_N u_0 = \varphi $ on $\partial \Omega$.

\item[B4.] Recall $c_1$ and $c_2$ in Section \ref{sec:2.1}. There exists a positive constant $\epsilon_2$ depending only on $n$ such that
\begin{equation} \label{eq:neu_condA4}
    \frac{c_2}{c_1} < \epsilon_2.
\end{equation}
\end{description}

\begin{theorem} \label{thm:4.1}
Assume {\bf B1} -- {\bf B4} and 
 $\epsilon_2$ is sufficiently small, then
there exists constant $C$ such that any solution $u(x,t) \in C^{3,2} (\overline{Q_T})$ of \eqref{eq:Neupro}
satisfies
\begin{equation}\label{eq:neugradest}
    \sup_{\overline{Q_T}} |Du| \leq C,
\end{equation}
where $C$ is independent of time.
\end{theorem}
Similar to Lemma \ref{lem:3.2}, we state the following Lemma:
\begin{lemma}\label{lem:4.2}
Let $u(x,t)$ be a smooth solution to \eqref{eq:Neupro}. There exists a constant 
$C=C(u_0)$ such that 
\begin{equation*}
    \sup_{Q_T} |u_t|^2=\sup_{\Omega_0} |u_t|^2 \leq C.
\end{equation*}

\end{lemma} 

With that, we continue to the\\
\textbf{Proof of Theorem 4.1:}
The idea of the proof is from \cite{MWW}. Firstly we define the following auxiliary function, 
\begin{equation} \label{eq:neuaux}
    \Psi(x) := \log{|DW|^2} + a_0 h, 
\end{equation}
where $a_0$ is a positive constant to be determined, $h$ is the same function defined on $\overline{\Omega}$ given by \eqref{eq:h_cond}, and 
\begin{equation*}
    W := u -Q,\quad\text{where $Q=- \varphi h$}
\end{equation*}
with $\varphi$ being the boundary function from \eqref{eq:Neupro}.
(Note that the $W$ here is not to be confused with the $W$ in \eqref{W.CAC}).
We remark that $\Psi$ is again a function of $Du$.  

Assume the maximum of $\Psi$ occurs at $(x_0,t_0)$. Similar to the proof of Theorem \ref{thm:3.1}, the proof consists of two cases: \\ \\
\textbf{Case 1}: $(x_0,t_0) \in \Gamma_T$. Again, the result in this case is independent of the governing equation of \eqref{eq:Neupro}. The proof in this case is essentially the same as in \cite{MWW} but again, for reader's convenience, we outline it here.

Many of the notations below have their analogues from the proof of Theorem \ref{thm:3.1}. Notice that from the boundary condition in \eqref{eq:Neupro}, we have $W_{x_n} \equiv 0$ on $\partial\Omega$ and hence $|DW|^2 = \sum_{i=1}^{n-1} W_{x_i}^2 $. Maximum principle gives  at $(x_0,t_0)$ that,
\begin{eqnarray*}
    0 \geq \,\Psi_{x_n} &=&
    \frac{ \sum_{i=1}^{n-1} 2W_{x_i} W_{x_i x_n} }{|DW|^2} - a_0 \\
   &= & \frac{ \sum_{i=1}^{n-1} 2 \left(  W_{x_i} W_{x_n x_i} + \sum_{j=1}^{n-1} k_{ij} W_{x_i} W_{x_j}  \right) }{|DW|^2} - a_0 \\
   &= & \frac{ \sum_{i,j=1}^{n-1} 2  k_{ij} W_{x_i} W_{x_j}   }{|DW|^2} - a_0 \\
   &\geq & 2 k_0 - a_0.
\end{eqnarray*}
Hence by taking $a_0 \in (0,2k_0)$, we have a contradiction, which means the maximum of $\Psi$ does not occur on $\Gamma_T$. \\ \\
\textbf{Case 2}: $(x_0,t_0)$ is a interior point of $\Omega_T$. We perform the following calculation at this point. Without loss of generality, we can assume $|Du|>1$. The maximum principle implies
\begin{equation} \label{eq:neumax1}
    \Psi_{x_i} = \frac{\left(|DW|^2\right)_{x_i}}{|DW|^2} + a_0 h_{x_i} = 0,
\end{equation}
\begin{equation}  \label{eq:neumax2}
    \Psi_{x_i x_j} = \frac{\left(|DW|^2 \right)_{x_i x_j}}{|DW|^2} - \frac{\left(|DW|^2 \right)_{x_i} \left(|DW|^2 \right)_{x_j}}{|DW|^4} + a_0 h_{x_i x_j}, \quad \{ \Psi_{x_i x_j} \} \leq 0,
\end{equation}
and
\begin{equation} \label{eq:neumax3}
    \Psi_t = \frac{\left(|DW|^2 \right)_t}{|DW|^2} \geq 0.
\end{equation}
Apply \eqref{eq:neumax1} to \eqref{eq:neumax2}, we have
\begin{equation}  \label{eq:neuPsiij}
    \Psi_{x_i x_j} = \frac{\left(|DW|^2 \right)_{x_i x_j}}{|DW|^2} - a_0^2 h_{x_i} h_{x_j} + a_0 h_{x_i x_j}.
\end{equation}
Since $\{a^{ij}\}$ is semi-positive, from \eqref{eq:neumax2} and \eqref{eq:neumax3} we have 
\begin{equation} \label{eq:neuineq1}
     0 \geq a^{ij} \Psi_{x_i x_j} - \Psi_t.
\end{equation}
With the result in \eqref{eq:neuPsiij}, \eqref{eq:neuineq1} becomes
\begin{equation}  \label{eq:neuineq2}
    0 \geq a^{ij} \frac{\left(|DW|^2 \right)_{x_i x_j}}{|DW|^2} - \frac{\left(|DW|^2 \right)_t}{|DW|^2} - a_0^2 a^{ij} h_{x_i} h_{x_j} + a_0 a^{ij} h_{x_i x_j}.
\end{equation}
Similar to the proof of Theorem \ref{thm:3.1}, we define
\begin{eqnarray}
   I_1 &:= & a^{ij} \frac{\left(|DW|^2 \right)_{x_i x_j}}{|DW|^2} - \frac{\left(|DW|^2 \right)_t}{|DW|^2}, \\
    I_2 &:= & -a_0^2 a^{ij} h_{x_i} h_{x_j}, \\
    I_3 &:= & a_0 a^{ij} h_{x_i x_j}     
\end{eqnarray}
Notice that $I_2$ and $I_3$ are estimated from below by \eqref{angI2est} and \eqref{angI3est}:
\begin{equation} \label{eq:neuIest}
    I_2 \geq - a_0^2 \left(\max_{\alpha} \tau^{\alpha \alpha} + \tau^{n n} \right),
\quad
    I_3 \geq a_0 k_1 \left(\sum_{\alpha} \tau^{\alpha \alpha} + \tau^{n n} \right).    
\end{equation}

Next we compute the expression $a^{ij} \left(|DW|^2 \right)_{x_i x_j} - \left(|DW|^2 \right)_t$ in $I_1$. 
Firstly, we differentiate the governing equation in \eqref{eq:Neupro}.
\begin{equation*}
    u_{x_k,t} = (a^{ij})_{x_k} u_{x_i x_j} + a^{ij} u_{x_i x_j x_k}.
\end{equation*}
Then 
\begin{eqnarray}
    a^{ij} \left(|DW|^2 \right)_{x_i x_j} - \left(|DW|^2 \right)_t 
    & =& 2 a^{ij} W_{x_k x_i} W_{x_k x_j} - 2 (a^{ij})_{x_k} W_{x_k} u_{x_i x_j} - 2 a^{ij} W_{x_k} Q_{x_i x_j x_k}
    \nonumber\\
    & =: & 2 J_1 + 2 J_2 + 2 J_3. \label{eq:neu_Ji}
\end{eqnarray}

For $J_3$, since $a^{ij}$ is bounded when $|Du| >1$, we simply estimate it by
\begin{equation}  \label{eq:neuJ3}
    J_3 \geq -C v.
\end{equation}

For $J_1$, notice that it is semi-positive definite.
Similar to the contact angle case, we will also make use of the orthonormal basis we have chosen for $\mathbb{R}^n$ and $n \times n$ symmetric matrices as in Section \ref{Sec:2.2}.  Firstly,
\begin{equation} \label{eq:neuJ1}
\begin{split}
   J_1 &\ = \ a^{ij} W_{x_k x_i} W_{x_k x_j} \\ 
   &\ =\ \text{tr}(D^2 W \  A \  D^2 W) \\
   &\ =\ \text{tr}(D^2 u \  A \  D^2 u) - 2 \text{tr}(D^2 u \  A \  D^2 Q)  + \text{tr}(D^2 Q \  A \  D^2 Q) \\
   &\ =\ J_{11} + J_{12} + J_{13},
\end{split}
\end{equation}
where 
$J_{11} :=\text{tr}(D^2 u \  A \  D^2 u)$,
$J_{12} :=  - 2 \text{tr}(D^2 u \  A \  D^2 Q)$,
$J_{13} :=  \text{tr}(D^2 Q \  A \  D^2 Q).$
Now we define the following semi-positive definite bilinear form on $n\times n$ symmetric matrices,
\begin{equation*}
    \mathcal{A}' (\cdot,\cdot) = \text{tr}( \cdot \  A \  \cdot).
\end{equation*}
Results for $\mathcal{A}'$ evaluated on the basis elements are
\begin{eqnarray*}
    \mathcal{A}' (\Phi^{\alpha \alpha},\Phi^{\alpha \alpha}) & = & \tau^{\alpha \alpha},\\
    \mathcal{A}' (\Phi^{\alpha n},\Phi^{\alpha n}) &=& \frac{1}{2} (\tau^{n n} + \tau^{\alpha \alpha}),\\
    \mathcal{A}' (\Phi^{n n},\Phi^{n n}) & = &\tau^{n n},\\
    \mathcal{A}' (\Phi^{\alpha \alpha},\Phi^{\alpha n}) &=& \frac{1}{\sqrt{2}} \tau^{\alpha n},\\
    \mathcal{A}' (\Phi^{\alpha \alpha},\Phi^{\beta \beta}) &=& 0, \quad (\alpha \neq \beta),\\
    \mathcal{A}' (\Phi^{\alpha \alpha},\Phi^{n n}) &=&0,\\
    \mathcal{A}' (\Phi^{\alpha \alpha},\Phi^{\beta n}) &=& 0 \quad (\alpha \neq \beta),\\
    \mathcal{A}' (\Phi^{\alpha n},\Phi^{\beta n}) &=& \frac{1}{2} \tau^{\alpha \beta} \quad (\alpha \neq \beta),\\
    \mathcal{A}' (\Phi^{\alpha n},\Phi^{n n}) &=& \frac{1}{\sqrt{2}} \tau^{\alpha n}.
\end{eqnarray*}
Notice that $J_{13}\geq 0$ and hence can be discarded.

For $J_{11}$, we decompose it as
\begin{equation} \label{eq:neuJ1_1}
\begin{split}
    J_{11} & = \mathcal{A}' (D^2 u \  A \  D^2 u) \\ & = \sum_{\alpha} (\gamma^{\alpha \alpha})^2 \tau^{\alpha \alpha} + \sum_{\alpha} \frac{(\gamma^{\alpha n})^2}{2} (\tau^{\alpha \alpha} + \tau^{n n}) + (\gamma^{n n})^2 \tau^{n n} \\
    & \ \ \ + \sum_{\alpha} \sqrt{2} \gamma^{\alpha \alpha} \gamma^{\alpha n} \tau^{\alpha n} + \sum_{\alpha \neq \beta} \frac{\gamma^{\alpha n} \gamma^{\beta n}}{2} \tau^{\alpha \beta} + \sum_{\alpha} \sqrt{2} \gamma^{\alpha n} \gamma^{n n} \tau^{\alpha n}.
\end{split}
\end{equation}

For $J_{12}$, we use $\gamma^{\alpha \alpha}$ to express $\gamma^{\alpha n}$ and $\gamma^{n n}$. Notice that \eqref{eq:neumax1} suggests
\begin{equation*}
    D^2 u \ DW = - \frac{a_0}{2} |DW|^2 Dh + D^2 Q \ DW,
\end{equation*}
or with respect to the $\phi^{\alpha}$ and $\phi^n$ directions, we have
\begin{equation} \label{eq:neu_phia}
    (\phi^{\alpha})^T D^2 u \ DW = - \frac{a_0}{2} |DW|^2 (\phi^{\alpha})^T Dh + (\phi^{\alpha})^T D^2 Q \ DW,
\end{equation}
\begin{equation} \label{eq:neu_phin}
    (\phi^n)^T D^2 u \ DW = - \frac{a_0}{2} |DW|^2 (\phi^n)^T Dh + (\phi^n)^T D^2 Q \ DW.
\end{equation}
Using $DW = Du - DQ$, where $Q =- \varphi h $, we define
\begin{equation} \label{eq:neu_rho_def}
\rho^{\alpha} := (\phi^{\alpha})^T DW = - (\phi^{\alpha})^T DQ,
\quad\text{and}\quad
\rho^n := (\phi^n)^T DW = |Du| - (\phi^n)^T DQ ,
\end{equation}
which lead to
 \begin{equation} \label{eq:neu_phi_rhoa}
    (\phi^{\alpha})^T D^2 u \ DW = \gamma^{\alpha \alpha} \rho^{\alpha} + \frac{ \gamma^{\alpha n} }{\sqrt{2}} \rho^n ,
\end{equation}
\begin{equation} \label{eq:neu_phi_rhon}
    (\phi^n)^T D^2 u \ DW = \sum_{\alpha} \frac{ \gamma^{\alpha n} }{\sqrt{2}} \rho^{\alpha}  + \gamma^{n n} \rho^n .
\end{equation}
Notice that:
\begin{equation} \label{eq:neu_rhoa}
    |\rho^{\alpha}| \leq \| Q  \|_{C^1(\overline{\Omega})},
\end{equation}
and for $|Du|$ sufficiently large, we have
\begin{equation} \label{eq:neu_rhon}
    \frac{3}{4} |Du| \leq \rho^n = (\phi^n)^T DW \leq \frac{5}{4} |Du|.
\end{equation}
Combining with \eqref{eq:neu_phia} and \eqref{eq:neu_phin}, we derive
\begin{eqnarray*}
    \gamma^{\alpha n} & = & \frac{\sqrt{2}}{\rho^n} \left( -\gamma^{\alpha \alpha} \rho^{\alpha}  - \frac{a_0}{2} |DW|^2 (\phi^{\alpha})^T Dh + (\phi^{\alpha})^T D^2 Q \ DW  \right) \\
    & = & - \frac{\sqrt{2} \rho^{\alpha}}{\rho^n} \gamma^{\alpha \alpha} + O(1) a_0 |Du| + O(1) \\
    \gamma^{n n} & = & \frac{1}{\rho^n} \left( -\sum_{\alpha} \frac{ \gamma^{\alpha n} }{\sqrt{2}} \rho^{\alpha}  - \frac{a_0}{2} |DW|^2 (\phi^n)^T Dh + (\phi^n)^T D^2 Q \ DW  \right) \\
    & = & -\sum_{\alpha} \frac{ \gamma^{\alpha n} \  \rho^{\alpha} }{\sqrt{2} \ \rho^n } + O(1) a_0 |Du| + O(1) \\
    & = & \sum_{\alpha} \frac{(\rho^{\alpha})^2}{ (\rho^n)^2 } \gamma^{\alpha \alpha} + \big( O(|Du| + O(1) \big) a_0 + O(1) + O\left( \frac{1}{|Du|} \right)
\end{eqnarray*}
Then $J_{12}$ can be bounded from below as
\begin{equation} \label{eq:neuJ1_2}
\begin{split}
    J_{12}   \geq &-2 \left\| D^2 Q \right\|  \left|\text{tr}(A \  D^2 Q)  \right| \\
     = & -2 \left\| D^2 Q \right\| \left|  \sum_{\alpha} \gamma^{\alpha \alpha} \text{tr}(A\Phi^{\alpha \alpha}) +  \sum_{\alpha} \gamma^{\alpha n} \text{tr}(A\Phi^{\alpha n}) + \gamma^{n n} \text{tr}(A\Phi^{n n})  \right| \\
     = & -2 \left\| D^2 Q \right\| \left|  \sum_{\alpha} \gamma^{\alpha \alpha} \tau^{\alpha \alpha} +  \sum_{\alpha} \sqrt{2} \gamma^{\alpha n} \tau^{\alpha n} + \gamma^{n n} \tau^{n n}  \right| \\
     \geq & -C \left|\sum_{\alpha} \gamma^{\alpha \alpha} \left(   \tau^{\alpha \alpha}  -   \frac{2\rho^{\alpha}}{\rho^n}  \tau^{\alpha n} +  \frac{(\rho^{\alpha})^2}{(\rho^n)^2}  \tau^{n n} \right) \right| \\
    &   -C  a_0 O\left( |Du| \right) \left(\sum_{\alpha} \sqrt{2} \tau^{\alpha n} + \tau^{n n} \right).
\end{split}
\end{equation}

Next we consider $J_2$ following the same process as in the proof of Theorem \ref{thm:3.1},
\begin{equation*}
\begin{split}
    J_2 = & - (a^{ij})_{x_k} u_{x_i x_j} W_{x_k} \\
    = & -D_{p_l} \left(|p| D^2_{p_i p_j} F \right)\Big|_{(Du,-1)} u_{x_i x_j} u_{x_k x_l} W_{x_k}   \\
    = & \  T^3(D^2 u , D^2 u \  DW).
\end{split}
\end{equation*}
From \eqref{eq:neu_phi_rhoa} and \eqref{eq:neu_phi_rhon} we have 
\begin{equation*}
    D^2 u \  DW = \left( \gamma^{\alpha \alpha} \rho^{\alpha} + \frac{ \gamma^{\alpha n} }{\sqrt{2}} \rho^n \right) \phi^{\alpha} + \left( \sum_{\alpha} \frac{ \gamma^{\alpha n} }{\sqrt{2}} \rho^{\alpha}  + \gamma^{n n} \rho^n  \right) \phi^n. 
\end{equation*}
We define, for $1\leq \alpha \leq n-1$, 
\begin{equation*}
    \eta^{\alpha} := \gamma^{\alpha \alpha} \rho^{\alpha} + \frac{ \gamma^{\alpha n} }{\sqrt{2}} \rho^n,
\,\,\,\text{and}\,\,\, 
    \eta^n := \sum_{\alpha} \frac{ \gamma^{\alpha n} }{\sqrt{2}} \rho^{\alpha}  + \gamma^{n n} \rho^n .
\end{equation*}
Then we can apply the result \eqref{eq:T3_cal1} and \eqref{eq:T3_cal2} to $J_2$ with $V = DW$. Now we estimate $J_1 + J_2 +J_3$ by the following decomposition:
\begin{equation*}
\begin{split}
    J_1 + J_2 +J_3 \ = & \ \frac{1}{10}\sum_{\alpha} (\gamma^{\alpha \alpha})^2 \tau^{\alpha \alpha} +J_{12} + J_{13} + J_3\\
   & \ + J_{11} - \frac{1}{10}\sum_{\alpha} (\gamma^{\alpha \alpha})^2 \tau^{\alpha \alpha} +J_2.
\end{split}
\end{equation*}
We claim that
\begin{equation} \label{eq:neu_cl1}
    \frac{1}{10}\sum_{\alpha} (\gamma^{\alpha \alpha})^2 \tau^{\alpha \alpha} +J_{12} + J_{13} + J_3 \geq -Cv,
\end{equation}
and
\begin{equation} \label{eq:neu_cl2}
    J_{11} - \frac{1}{10}\sum_{\alpha} (\gamma^{\alpha \alpha})^2 \tau^{\alpha \alpha} +J_2 \geq 0.
\end{equation}

Recall the estimates \eqref{eq:neuJ1_2} and \eqref{eq:neuJ3} for $J_{12}$ and $J_3$. 
Note that $\frac{1}{10}\sum_{\alpha} (\gamma^{\alpha \alpha})^2 \tau^{\alpha \alpha}$ is a positive quadratic form on $\gamma^{\alpha \alpha}$, $1\leq \alpha \leq n-1$, while $J_{12}$ is bounded below by a linear form on $\gamma^{\alpha \alpha}$. Hence \eqref{eq:neu_cl1} follows directly.

Now consider claim \eqref{eq:neu_cl2}. With \eqref{eq:neuJ1_1} and results in Section \ref{Sec:2.2}, \eqref{eq:T3_cal1} and \eqref{eq:T3_cal2}, we analyze the left hand side of \eqref{eq:neu_cl2} by writing it as a bilinear form of 
\[\gamma = \left(\gamma^{11}, \cdots, \gamma^{n-1,n-1}, \gamma^{1n}, \cdots, \gamma^{n-1,n}, \gamma^{nn}  \right)^T.\] 
We have
\begin{equation}
\begin{split}
    \,\,\, & J_{11} - \frac{1}{10}\sum_{\alpha} (\gamma^{\alpha \alpha})^2 \tau^{\alpha \alpha} +J_2 \\
  = \,\,\,&  \sum_{\alpha} \left( \frac{9}{10}  \tau^{\alpha \alpha} + \rho^{\alpha} T^3 (\phi^{\alpha}, \phi^{\alpha}, \phi^{\alpha}) \right)   (\gamma^{\alpha \alpha})^2  \\
  & + \sum_{\alpha \neq \beta}  \rho^{\beta} T^3 (\phi^{\alpha}, \phi^{\alpha}, \phi^{\beta})    \gamma^{\alpha \alpha} \gamma^{\beta \beta} \\
  & + \sum_{\alpha}    \left( \rho^n T^3 (\phi^{\alpha}, \phi^{n}, \phi^{\alpha}) + \frac{1}{2} \tau^{\alpha \alpha} +  \frac{1}{2} \tau^{nn} + \rho^{\alpha} T^3 (\phi^{\alpha}, \phi^n, \phi^n)  \right) (\gamma^{\alpha n})^2 \\
  & + \sum_{\alpha \neq \beta}  \left( \rho^n  T^3 (\phi^{\alpha}, \phi^{n}, \phi^{\beta}) + \frac{1}{2}  \tau^{\alpha \beta} + \rho^{\beta} T^3 (\phi^{\alpha}, \phi^n, \phi^n)   \right)  \gamma^{\alpha n} \gamma^{\beta n} \\
  & + \sum_{\alpha}  \left( \frac{\rho^n}{\sqrt{2}}  T^3 (\phi^{\alpha}, \phi^{\alpha}, \phi^{\alpha}) + \sqrt{2}  \tau^{\alpha n}  + \frac{\rho^{\alpha}}{\sqrt{2}} T^3 (\phi^{\alpha}, \phi^{\alpha}, \phi^{n}) + \sqrt{2} \rho^{\alpha}  T^3 (\phi^{\alpha}, \phi^{n}, \phi^{\alpha}) \right) \gamma^{\alpha \alpha} \gamma^{\alpha n} \\
  & + \sum_{\alpha \neq \beta}  \left( \frac{\rho^n}{\sqrt{2}}  T^3 (\phi^{\alpha}, \phi^{\alpha}, \phi^{\beta}) + \frac{\rho^{\beta}}{\sqrt{2}} T^3 (\phi^{\alpha}, \phi^{\alpha}, \phi^{n}) +  \sqrt{2} \rho^{\alpha}  T^3 (\phi^{\beta}, \phi^{n}, \phi^{\alpha})  \right) \gamma^{\alpha \alpha} \gamma^{\beta n} \\
  & + \left( \rho^n  T^3 (\phi^{n}, \phi^{n}, \phi^{n})  +   \tau^{n n}  \right)  (\gamma^{n n})^2 \\
  & + \sum_{\alpha} \left( \sqrt{2}  \rho^n  T^3 (\phi^{\alpha}, \phi^n, \phi^n) + \frac{\rho^n}{\sqrt{2}} T^3 (\phi^{n}, \phi^{n}, \phi^{\alpha}) + \sqrt{2} \tau^{\alpha n}  +  \frac{\rho^{\alpha}}{\sqrt{2}} T^3 (\phi^{n}, \phi^{n}, \phi^{n})  \right) \gamma^{\alpha n} \gamma^{n n}  \\
  & + \sum_{\alpha} \left(  \rho^n  T^3 (\phi^{\alpha}, \phi^{\alpha}, \phi^{n}) + \rho^{\alpha} T^3 (\phi^{n}, \phi^{n}, \phi^{\alpha})   \right) \gamma^{\alpha \alpha} \gamma^{n n} \\
  :=\,\,\, & \mathcal{B}'(\gamma, \gamma).
\end{split}
\end{equation}
Here we define the same matrix representation for $\mathcal{B}'(\cdot,\cdot)$ as in \eqref{eq:B_matrep}. Recall Lemma \ref{lem:2.4} and estimates for $\rho^{\alpha}$ and $\rho^n$, \eqref{eq:neu_rhoa} and \eqref{eq:neu_rhon}. When $\epsilon_2$ is sufficiently small, all the diagonal entries of $\mathcal{B}'$ are bounded from below in the sense that there exists positive $c_1''(\varphi,c_1,c_2)$ such that
\begin{equation*}
    B'^{\alpha\alpha,\alpha\alpha}, \ B'^{\alpha n,\alpha n} \geq c_1'', \quad B'^{nn,nn} \geq \frac{c_1''}{v^2}.
\end{equation*}
Furthermore, non-diagonal entries of $\mathcal{B}'$ are bounded from above
in the sense that there exists positive $c_2''(\theta,c_1,c_2)$ such that
\begin{equation*}
    B'^{\alpha\alpha,\beta\beta}, \ B'^{\alpha \alpha, \alpha n} , \ B'^{\alpha \alpha,\beta n}, \  B'^{\alpha n,\beta n} \leq c_2'', \quad B'^{\alpha \alpha, n n} , \  B'^{\alpha n, n n} \leq \frac{c_2''}{v^2},\quad\text{for $\alpha\neq\beta$.}
\end{equation*}
Hence, there exists sufficiently small $\epsilon_2$ depending on $n$ such that $\mathcal{B}'(\cdot,\cdot)$ is semi-positive definite. This proves the
second claim \eqref{eq:neu_cl2}.

Finally, \eqref{eq:neu_cl1} and \eqref{eq:neu_cl2} suggest
\begin{equation*}
    a^{ij} \left(|DW|^2 \right)_{x_i x_j} - \left(|DW|^2 \right)_t \geq -C v,  
\quad
\text{or}
\quad
    I_1 \geq - \frac{C v }{|DW|^2} .
\end{equation*}
Combining this lower bound for $I_1$ together with \eqref{eq:neuIest} for $I_2$ and $I_3$, 
we obtain from \eqref{eq:neuineq2} that 
\begin{equation*}
\begin{split}
    0 & \geq a^{ij} \Psi_{x_i x_j} - \Psi_t \\
    & \geq - \frac{Cv }{|Dw|^2} - a_0^2 \left(\max_{\alpha} \tau^{\alpha \alpha} + \tau^{n n} \right) + a_0 k_1 \left(\sum_{\alpha} \tau^{\alpha \alpha} + \tau^{n n} \right) \\
    & \geq - \frac{C}{v} - a_0^2 \left(\max_{\alpha} \tau^{\alpha \alpha} + \tau^{n n} \right) + a_0 k_1 \left(\sum_{\alpha} \tau^{\alpha \alpha} + \tau^{n n} \right).
\end{split}
\end{equation*}
By choosing $a_0=\min \left\{ k_1 \frac{\sum_{\alpha} \tau^{\alpha \alpha} + \tau^{n n}}{\left(\max_{\alpha} \tau^{\alpha \alpha} + \tau^{n n} \right)},\frac{3}{2} k_0  \right\}$, this implies at $(x_0,t_0)$ that
\begin{equation*}
    v\leq C
\end{equation*}
completing the proof of the gradient estimate \eqref{eq:neugradest}.
\hfill$\square$\\

Similar to the contact angle boundary problem, the next step is to establish a solution to the elliptic problem, and finally to show that the solution to the parabolic problem converges to a translating solution. Consider \eqref{eq:elneu_pro}:
\begin{equation} \label{eq:neu_el_pro}
    \left\{ \begin{array}{cllcl}
           \lambda  & = & \sqrt{1+|Dw|^2} \  D^2_{p_i p_j} F (Dw,-1) w_{x_i x_j}    \quad &  \text{in} &  \ \Omega , \\
            D_N w & = & \varphi  \quad &  \text{on} & \ \partial \Omega .  
    \end{array}  \right.
\end{equation}
\begin{theorem}
    Assume {\bf B1} -- {\bf B4} and 
 $\epsilon_2$ is sufficiently small. Then there exists a unique $\lambda$ and a corresponding smooth solution $w$ for \eqref{eq:neu_el_pro}. In particular, the solution is unique up to a constant. 
\end{theorem}
\begin{proof}
    Similar to the proof in Theorem \ref{thm:3.3}, we firstly consider the following eigenvalue problem: 
\begin{equation} \label{eq:neu_el_eigpro}
    \left\{ \begin{array}{cllcl}
           \epsilon w^{\epsilon}  & = & \sqrt{1+|Dw^{\epsilon}|^2} \  D^2_{p_i p_j} F (Dw^{\epsilon},-1) w^{\epsilon}_{x_i x_j}    \quad &  \text{in} &  \ \Omega , \\
            D_N w^{\epsilon} & = & \varphi  \quad  & \text{on}  & \ \partial \Omega .
    \end{array}  \right.
\end{equation}
As in \cite{MWW}, we can obtain the existence of the solution to \eqref{eq:neu_el_eigpro} and a uniform bound to $| \epsilon w^{\epsilon}|$. Then we also apply the same technique as in Theorem \ref{thm:4.1} to the gradient estimate of the solution to \eqref{eq:neu_el_eigpro}. 

Define the same auxiliary function $\Psi(x) = \log{|DW|^2} + a_0 h$, where
$W = w^{\epsilon} + \varphi h$. Then we consider two cases for the maximum point $x_0$ of $\Psi(x)$. The first one is that $x_0$ is on the boundary $\partial \Omega$, for which we can follow the proof in Theorem \ref{thm:4.1}, case 1. The second one is that $x_0$ is in the interior of $\Omega$. Most steps also follow as in Theorem \ref{thm:4.1}, case 2.

Instead of \eqref{eq:neuineq1} in Theorem \ref{thm:4.1}, we only make use of 
\begin{equation} \label{neu_el_max}
    0 \geq a^{ij} \Psi_{x_i x_j} = 
    a^{ij} \frac{\left(|DW|^2\right)_{x_i x_j}}{|DW|^2}  - a_0^2 a^{ij} h_{x_i} h_{x_j} + a_0 a^{ij} h_{x_i x_j} .
\end{equation}
By differentiating the governing function in \eqref{eq:neu_el_eigpro}, we have
\begin{equation}
    a^{ij} \left(|DW|^2 \right)_{x_i x_j} = 2 \left(J_1 + J_2 +J_3 + 2 \epsilon W_{x_k} w^{\epsilon}_{x_k} \right),
\end{equation}
where $J_i's$ are identical to \eqref{eq:neu_Ji} by substituting $u$ with $w^{\epsilon}$. And we have
\begin{equation*}
    W_{x_k} w^{\epsilon}_{x_k} = \left(w^{\epsilon}_{x_k} + (\varphi h)_{x_k} \right) w^{\epsilon}_{x_k} \geq 0,
\end{equation*}
provided $|Dw^{\epsilon}|$ is sufficiently large. Hence we also obtain
\begin{equation*}
    a^{ij} \left(|DW|^2 \right)_{x_i x_j} \geq -C |Dw^{\epsilon}|.
\end{equation*}
Then the gradient estimate for $w^{\epsilon}$ follows by the same approach as in the proof of Theorem \ref{thm:4.1}. 

Similar to Theorem \ref{thm:3.3}, we have $\epsilon w^{\epsilon} \rightarrow \lambda$ and $\lambda$ is unique.
\end{proof}

Finally we state the asymptotic result of the solution to \eqref{eq:Neupro}, which can be proved by the same approach as in the contact angle boundary problem. We omit the proof here.

\begin{corollary}
    For any solution $u(x,t)$ to \eqref{eq:Neupro}, there exists a positive constant $C$ independent of time, such that
    \begin{equation*}
        |u(x,t)-\lambda t| \leq C.
    \end{equation*}
\end{corollary}
\noindent
\begin{theorem}
Let $u$ be any solution of \eqref{eq:Neupro}. Then
    $u(x,t)-(w(x)+\lambda t)$ converges to a constant as $t\longrightarrow+\infty$, where $w$ is a solution of \eqref{eq:neu_el_pro}.
\end{theorem}

\bibliographystyle{plain}
\bibliography{references}

\begin{bibdiv}
\begin{biblist}

\bib{AW1}{article}{
      author={Altschuler, Steven~J.},
      author={Wu, Lang~F.},
       title={Convergence to translating solutions for a class of quasilinear parabolic boundary problems},
        date={1993-01},
        ISSN={1432-1807},
     journal={Mathematische Annalen},
      volume={295},
      number={1},
       pages={761–765},
         url={http://dx.doi.org/10.1007/BF01444916},
}

\bib{AW2}{article}{
      author={Altschuler, Steven~J.},
      author={Wu, Lang~F.},
       title={Translating surfaces of the non-parametric mean curvature flow with prescribed contact angle},
        date={1994-01},
        ISSN={1432-0835},
     journal={Calculus of Variations and Partial Differential Equations},
      volume={2},
      number={1},
       pages={101–111},
         url={http://dx.doi.org/10.1007/BF01234317},
}

\bib{AGLM}{article}{
      author={Alvarez, Luis},
      author={Guichard, Fr{\'e}d{\'e}ric},
      author={Lions, Pierre~Louis},
      author={Morel, Jean~Michel},
       title={Axioms and fundamental equations of image processing},
        date={1993},
     journal={Archive for rational mechanics and analysis},
      volume={123},
       pages={199\ndash 257},
}

\bib{alvarez1992image}{article}{
      author={Alvarez, Luis},
      author={Lions, Pierre-Louis},
      author={Morel, Jean-Michel},
       title={Image selective smoothing and edge detection by nonlinear diffusion (ii)},
        date={1992},
     journal={SIAM Journal on Numerical Analysis},
      volume={29},
      number={3},
       pages={845\ndash 866},
}

\bib{An1}{article}{
      author={Andrews, Ben},
       title={Volume-preserving anisotropic mean curvature flow},
        date={2001},
     journal={Indiana University mathematics journal},
       pages={783\ndash 827},
}

\bib{CNS}{article}{
      author={Caffarelli, Luis},
      author={Nirenberg, Louis},
      author={Spruck, Joel},
       title={The dirichlet problem for nonlinear second order elliptic equations, iii: Functions of the eigenvalues of the hessian},
        date={1985},
}

\bib{cahn1993equilibrium}{article}{
      author={Cahn, John~W},
      author={Handwerker, Carol~A},
       title={Equilibrium geometries of anisotropic surfaces and interfaces},
        date={1993},
     journal={Materials Science and Engineering: A},
      volume={162},
      number={1-2},
       pages={83\ndash 95},
}

\bib{CKN}{article}{
      author={Cesaroni, Annalisa},
      author={Kr{\"o}ner, Heiko},
      author={Novaga, Matteo},
       title={Graphical translators for anisotropic and crystalline mean curvature flow},
        date={2022},
     journal={Journal of Mathematical Analysis and Applications},
      volume={514},
      number={2},
       pages={126314},
}

\bib{Cl}{article}{
      author={Clutterbuck, Julie},
       title={Interior gradient estimates for anisotropic mean-curvature flow},
        date={2007},
     journal={Pacific Journal of Mathematics},
      volume={229},
      number={1},
       pages={119\ndash 136},
}

\bib{CY1}{article}{
      author={Cui, Can},
      author={Yip, Nung~Kwan},
       title={Two dimensional anisotropic mean curvature flow with contact angle condition},
        note={Available from \url{http://sites.google.com/view/cancui}, to be submitted},
}

\bib{de2003capillarity}{book}{
      author={de~Gennes, P.G.},
      author={Brochard-Wyart, F.},
      author={Quere, D.},
       title={Capillarity and wetting phenomena: Drops, bubbles, pearls, waves},
   publisher={Springer New York},
        date={2003},
        ISBN={9780387005928},
         url={https://books.google.com/books?id=MxLQk8vms-kC},
}

\bib{desbrun1999implicit}{inproceedings}{
      author={Desbrun, Mathieu},
      author={Meyer, Mark},
      author={Schr{\"o}der, Peter},
      author={Barr, Alan~H},
       title={Implicit fairing of irregular meshes using diffusion and curvature flow},
        date={1999},
   booktitle={Proceedings of the 26th annual conference on computer graphics and interactive techniques},
       pages={317\ndash 324},
}

\bib{Ecker2004}{book}{
      author={Ecker, Klaus},
       title={Regularity theory for mean curvature flow},
   publisher={Birkh\"{a}user Boston},
        date={2004},
        ISBN={9780817682101},
         url={http://dx.doi.org/10.1007/978-0-8176-8210-1},
}

\bib{EH1}{article}{
      author={Ecker, Klaus},
      author={Huisken, Gerhard},
       title={Mean curvature evolution of entire graphs},
        date={1989},
     journal={Annals of Mathematics},
      volume={130},
      number={3},
       pages={453\ndash 471},
}

\bib{finn1986equilibrium}{book}{
      author={Finn, Robert},
       title={Equilibrium capillary surfaces},
      series={Grundlehren der mathematischen Wissenschaften},
   publisher={Springer-Verlag},
     address={New York},
        date={1986},
      volume={284},
        ISBN={978-0387962418},
}

\bib{GMWW}{article}{
      author={Gao, Zhenghuan},
      author={Ma, Xinan},
      author={Wang, Peihe},
      author={Weng, Liangjun},
       title={Nonparametric mean curvature flow with nearly vertical contact angle condition},
        date={2021},
     journal={J. Math. Study},
      volume={54},
      number={1},
       pages={28\ndash 55},
}

\bib{garcke1998anisotropic}{article}{
      author={Garcke, Harald},
      author={Nestler, Britta},
      author={Stoth, Barbara},
       title={On anisotropic order parameter models for multi-phase systems and their sharp interface limits},
        date={1998},
     journal={Physica D: Nonlinear Phenomena},
      volume={115},
      number={1-2},
       pages={87\ndash 108},
}

\bib{Hu1}{article}{
      author={Huisken, Gerhard},
       title={Non-parametric mean-curvature evolution with boundary-conditions},
        date={1989},
     journal={Journal of differential equations},
      volume={77},
      number={2},
       pages={369\ndash 378},
}

\bib{MWW}{article}{
      author={Ma, Xi-Nan},
      author={Wang, Pei-He},
      author={Wei, Wei},
       title={Constant mean curvature surfaces and mean curvature flow with non-zero neumann boundary conditions on strictly convex domains},
        date={2018},
        ISSN={0022-1236},
     journal={Journal of Functional Analysis},
      volume={274},
      number={1},
       pages={252\ndash 277},
         url={https://www.sciencedirect.com/science/article/pii/S0022123617303786},
}

\bib{Mu}{article}{
      author={Mullins, William~W},
       title={Two-dimensional motion of idealized grain boundaries},
        date={1956},
     journal={Journal of Applied Physics},
      volume={27},
      number={8},
       pages={900\ndash 904},
}

\bib{sochen1998general}{article}{
      author={Sochen, Nir~A},
      author={Kimmel, Ron},
      author={Malladi, Ramu},
       title={A general framework for low-level vision},
        date={1998},
     journal={IEEE Transactions on Image Processing},
      volume={7},
      number={3},
       pages={310\ndash 318},
}

\bib{Ta}{article}{
      author={Taylor, Jean~E},
       title={Ii—mean curvature and weighted mean curvature},
        date={1992},
     journal={Acta metallurgica et materialia},
      volume={40},
      number={7},
       pages={1475\ndash 1485},
}

\bib{taylor1993motion}{inproceedings}{
      author={Taylor, Jean~E},
       title={Motion of curves by crystalline curvature, including triple junctions and boundary points},
        date={1993},
   booktitle={Proc. symp. pure math},
      volume={54},
       pages={417\ndash 438},
}

\end{biblist}
\end{bibdiv}

\end{document}